\documentclass{amsart}

\usepackage{amsthm,amsfonts,amsmath,amssymb,mathrsfs,bbold,amscd}

\newtheorem{thm}{Theorem}[subsection]
\newtheorem{lem}[thm]{Lemma}
\newtheorem{prop}[thm]{Proposition}
\newtheorem{cor}[thm]{Corollary}
\newtheorem{defn}[thm]{Definition}

\newenvironment{rmk}{\refstepcounter{thm} \medskip \noindent {\bf  Remark \arabic{section}.\arabic{subsection}.\arabic{thm}.\,}}{\hfill\mbox{}\bigskip}

\newtheorem{thmint}{Theorem}

\newtheorem{corint}[thmint]{Corollary}

\newcounter{num}

\newenvironment{thmlist}{\begin{list}{(\roman{num})}{\usecounter{num}\setlength{\leftmargin}{25pt}
\setlength{\itemindent}{0pt}\setlength{\labelwidth}{20pt}\setlength{\labelsep}{5pt}\setlength{\itemsep}{0in}}}{\end{list}}

\def\fg{{\mathfrak g}}

\def\fk{{\mathfrak k}}

\def\fp{{\mathfrak p}}

\def\fz{{\mathfrak z}}

\def\B{\mathbb{B}}
\def\C{\mathbb{C}}

\def\N{\mathbb{N}}

\def\R{\mathbb{R}}

\def\i{\sqrt{-1}}

\def\del{\partial}

\newcommand{\cD}{\mathcal{D}}
\newcommand{\cE}{\mathcal{E}}
\newcommand{\cF}{\mathcal{F}}
\newcommand{\cG}{\mathcal{G}}
\newcommand{\cH}{\mathcal{H}}

\newcommand{\cK}{\mathcal{K}}
\newcommand{\cL}{\mathcal{L}}
\newcommand{\cM}{\mathcal{M}}

\newcommand{\cO}{\mathcal{O}}

\newcommand{\cS}{\mathcal{S}}

\newcommand{\re}{\operatorname{Re}}

\newcommand{\Aut}{\operatorname{Aut}}
\newcommand{\aut}{\operatorname{\mathfrak{aut}}}
\newcommand{\Cal}{\operatorname{Cal}}

\newcommand{\End}{\operatorname{End}}
\newcommand{\sF}{\mathscr{F}_{\xi}}
\newcommand{\fol}{\operatorname{\mathfrak{fol}}}
\newcommand{\Fol}{\operatorname{Fol}}

\newcommand{\Ham}{\mathfrak{h}^{\operatorname{Ham}}}
\newcommand{\hol}{\operatorname{\mathfrak{hol}}}

\newcommand{\im}{\operatorname{Im}}

\newcommand{\Li}{\operatorname{L}}

\newcommand{\PSH}{\operatorname{PSH}}
\newcommand{\Ric}{\operatorname{Ric}}

\newcommand{\tr}{\operatorname{tr}}
\newcommand{\Vol}{\operatorname{Vol}}

\newcommand{\contr}{\,\lrcorner\,}

\newcommand{\ol}[1]{\overline{#1}}

\begin{document}

\title[uniqueness of Sasaki-extremal metrics]
{Monge-Amp\`{e}re operators, energy functionals, and uniqueness of Sasaki-extremal metrics}
\author[C. van Coevering]{Craig van Coevering}
\address{School of Mathematical Sciences, U.S.T.C., Anhui, Hefei 230026, P. R. China}
\curraddr{The Mathematical Sciences Research Institute, 17 Gauss Way, Berkeley, CA 94720-5070}
\email{craigvan@ustc.edu.cn}

\begin{abstract}
We develop some pluripotential theoretic techniques for the transversally holomorphic foliation of a Sasakian manifold.
We prove the convexity of the K-energy along weak geodesics for Sasakian manifolds.  This implies that the K-energy is
bounded below if a constant scalar curvature structure exists with those metrics minimizing it.  More generally, a
relative version of the K-energy is convex, and bounded below if there exists a Sasaki-extremal metric, providing
an important necessary condition for Sasaki-extremal metrics.  Another application is a proof of the uniqueness of Sasaki-extremal metrics for a fixed transversally holomorphic structure on the Reeb foliation.

\end{abstract}

\maketitle

\section{Introduction}\label{sec:intro}

There has been a renewed interest in Sasakian geometry recently from two sources.  First, they have provided a very good
source of new examples of Einstein manifolds~\cite{BGK05,BoGa05,Kol05} and the survey article~\cite{Spa11}.
Second, they play a crucial role in the AdS/CFT correspondence~\cite{AFHS98,Mal98,MoPl99,MSY08},
which is a proposed duality between string theory on an odd dimensional Einstein manifold and conformal field theory.
It is also worth mentioning that metric cones over Sasakian manifolds arise as the tangent cones at infinity of
non-compact Calabi-Yau manifolds with Euclidean volume growth~\cite{DoSu14,CoMi14}.

These new results have been facilitated by the fact that
a Sasakian manifold is an odd dimensional contact analogue of a K\"{a}hler manifold, both the metric cone over the manifold
and the transversal space to the Reeb foliation have natural K\"{a}hler structures, so many of the techniques used
in K\"{a}hler geometry are applicable.  In particular, one expects that much of the results in K\"{a}hler geometry
related to the program proposed by S. Donaldson~\cite{Don97,Don99}, which was conjectured earlier by S.-T. Yau~\cite{Yau93},
will hold for Sasakian manifolds, in which the existence and uniqueness of constant scalar curvature K\"{a}hler metrics is
considered as a problem in infinite dimensional geometric invariant theory.  Much of the work was done earlier and independently
by T. Mabuchi, and S. Semmes~\cite{Mab87,Sem92}, in which the space of K\"{a}hler metrics $\cH$ in a given K\"{a}hler class
was shown to have a natural weak Riemannian structure and Riemannian connection.
The role of the Kempf-Ness functional in finite dimensional
geometric invariant theory is played by the K-energy on $\cH$.  It was observed by S. Donaldson that the existence of
geodesics in $\cH$ would lead to a proof of uniqueness of constant scalar curvature metrics in $\cH$ and
some sort of convexity of the K-energy should provide necessary and sufficient for existence.
Unfortunately, smooth geodesics are not known to exist in $\cH$.  But X. X. Chen proved
the existence of weak $C^{1,1}$ geodesics~\cite{Che00a}.  This was sufficient for X. X. Chen to
prove the uniqueness of constant scalar curvature K\"{a}hler metrics when $c_1(M)\leq 0$. Uniqueness was proved in general
by X. X. Chen and G. Tian~\cite{ChTi08} by proving stronger partial regularity on the geodesics.  Recently,
R. Berman and B. Berndtsson~\cite{BeBer14} and X. X. Chen, L. Li, and M. Paun~\cite{CLP14} proved the geodesic convexity
of the K-energy on weak geodesics, giving a simpler proof of uniqueness of constant scalar curvature K\"{a}hler metrics.

In Sasakian geometry the Reeb vector field is the analogue of a polarization in K\"{a}hler geometry.
And $\cH$ is the space of transversal K\"{a}hler metrics with the given polarization and transversal complex structure.
P. Guan and X. Zhang~\cite{GuaZha12} proved the existence of weak $C^{1,1}$ geodesics between elements of $\cH$.  They also
proved the uniqueness of constant scalar curvature Sasaki(cscS) metrics when $c_1^b(M)\leq 0$, where $c_1^b(M)$ denotes
the basic first Chern class of the Reeb foliation.  In the Sasaki-Einstein case $c_1^b(M) =a[\omega^T], a>0$,
Y. Nitta and K. Sekiya proved the uniqueness of Sasaki-Einstein metrics, up to automorphisms of the transversal
holomorphic structure, by extending the arguments of S. Bando and T. Mabuchi~\cite{NiSe12}.  Uniqueness for toric
cscS structures is also known due to K. Cho, A. Futaki, K. Ono~\cite{CFO08}.

In this article we prove the uniqueness of cscS metrics in $\cH$ in general, and more generally, we prove the uniqueness
of Sasaki-extremal metrics up to the automorphisms of the Reeb foliation and its transversal holomorphic structure.
Sasaki-extremal metrics were first defined by C. Boyer, K. Galicki, S. Simanca~\cite{BoGaSi08}, termed
\emph{canonical Sasakian} metrics.
A Sasaki-extremal metric is a critical point of the \emph{Calabi functional}
\[\Cal_{M,\xi} :\cH \rightarrow\R \]
\begin{equation}
\Cal_{M,\xi}(\phi) := \int_M (S_{\phi} -\ol{S})^2 \, d\mu_\phi,
\end{equation}
where $d\mu_\phi =(\omega^T +dd^c \phi)^m \wedge\eta$.  As in the K\"{a}hler
case, extremal metrics are constant scalar curvature precisely when the transversal Futaki invariant vanishes.
Thus it enlarges the cases in which a \emph{canonical} metric exists.  There has been much research on
Sasaki-extremal metrics recently.  See~\cite{Boy11,BoTF13,BoTF14,BoTF15} for some recent work.

This article will provide the useful uniqueness result and obstructions involving the K-energy.
Thus, when they exist Sasaki-extremal metrics provide a canonical Sasakian metric for a given
transversely holomorphic foliation.  But from work in the K\"{a}hler case, we know that such metrics will not always exist.

The central result is convexity of the K-energy along weak $C^{1,1}$ geodesics, denoted $C^{1,1}_w$.
See the definition before Theorem~\ref{thm:weak-geo}.
As in~\cite{Che00b} we can extend the
K-energy $\cM:\cH \rightarrow\R$ to $\cH_{1,1}$, where $\cH_{1,1}$ is the space of transversal K\"{a}hler potentials
$\phi\in C^{1,1}_w$, weak $C^{1,1}$, with $\omega^T + dd^c \phi \geq 0$.

Let $\phi_0,\phi_1 \in\cH$, where we consider $\cH$ to be the space of smooth transversal K\"{a}hler potentials.
Let $\phi_t, 0\leq t\leq 1,$ be a weak $C^{1,1}$ geodesic, that is $\phi \in C^{1,1}_w(M\times [0,1])$ and
$\omega^T +dd^c \phi_t \geq 0$ for each $t\in[0,1]$.
\begin{thmint}\label{thmint:conv-K-ener}
The K-energy $\cM$ is convex along weak $C_w^{1,1}$ geodesics, that is, $\cM(\phi_t)$ is convex in $t\in[0,1]$.
\end{thmint}

The proof of Theorem~\ref{thmint:conv-K-ener} involves pluripotential theoretic arguments on the transversal space
to the Reeb foliation.  Much of \S~\ref{sec:Sasak-trans} is spent developing the necessary background
on transversal plurisubharmonic functions, currents, and Monge-Anp\`{e}re operators on the transversal space.
Much of this work is of independent interest, such as weak continuity of the transversal Monge-Anp\`{e}re operator
and a strong maximal principle.  The latter give uniqueness for weak geodesics that are only assumed to
be continuous.  These results hopefully will provide a useful framework for future work in Sasakian geometry
along the lines of the analytical approaches to K\"{a}hler geometry such as~\cite{BBGZ13,Ber15}.

In \S~\ref{sec:en-funct} we define the energy functionals on the space of potentials that will be
needed to define the K-energy $\cM$ on weak potentials and in proving Theorem~\ref{thmint:conv-K-ener}.
Theorem~\ref{thmint:conv-K-ener} is proved in \S~\ref{subsec:con-K-en}.  The main part of the proof
is proving that $\cM(u_\tau)$ is weakly subharmonic in $\tau\in D\subset\C$ when $\{ u_\tau \}$ is a
weak geodesic in the domain $D$.

An important application of Theorem~\ref{thmint:conv-K-ener} is the proof of uniqueness of constant scalar curvature
Sasakian (cscS) structures modulo diffeomorphisms preserving the transversely holomorphic foliation.
We denote by $\cS(\xi,J)$ the space of Sasakian structures with Reeb vector field $\xi$ and transversely holomorphic structure $J$.
Let $\Fol(\mathscr{F}_\xi,J)$ be the group of diffeomorphisms preserving the Reeb foliation $\mathscr{F}_\xi$ along
with its transversely holomorphic structure.
\begin{corint}\label{corint:unique-cscS}
Suppose that $(\eta_0,\xi,\omega^T_0),(\eta_1,\xi,\omega^T_1) \in\cS(\xi,J)$ are two cscS structures.  Then there is
a $g\in\Fol(\mathscr{F}_\xi,J)$ so that $g^*\omega^T_1 =\omega_0^T$.
\end{corint}

Using basic properties of convex functions we easily prove the following sub-slope inequality.
\begin{corint}\label{cor:Mab-bound}
Suppose $\phi_0,\phi_1 \in\cH$, then the following inequality holds
\[ \cM(\phi_1) -\cM(\phi_0) \geq -d(\phi_0,\phi_1)\bigl(\Cal_{M,\xi}(\phi_0)\bigr)^{\frac{1}{2}}, \]
where $d$ is the distance function of the Mabuchi metric on $\cH$.
\end{corint}
Thus any metric with constant scalar curvature minimizes the K-energy.  Furthermore, by Corollary~\ref{corint:unique-cscS}
in this case the K-energy achieves its minimum precisely on the orbit of $\Fol(\mathscr{F}_\xi,J)$.

More generally we consider Sasaki-extremal structures.
When considering Sasaki-extremal structures it is useful to consider a modified or relative version of
the K-energy $\cM^V$.  Let $G\subset\Fol(\mathscr{F}_\xi,J)$ be a maximal compact connected subgroup,
and $(g,\eta,\xi,\Phi)$ be $G$-invariant.  Then $\cM^V$ is restricted to the space $\cH^G$ of $G$-invariant
potentials and has critical point precisely the potential corresponding to Sasaki-extremal structures.
Here $V$ denotes the extremal vector field which is a transversely holomorphic vector field which depends
only on the choice of maximal compact group.

Using the convexity of $\cM^V$ along weak geodesics we are able to prove the uniqueness of Sasaki-extremal structures
modulo $\Fol(\mathscr{F}_\xi,J)$.
\begin{corint}\label{corint:unique-Sasak-ext}
Suppose that $(\eta_0,\xi,\omega^T_0),(\eta_1,\xi,\omega^T_1) \in\cS(\xi,J)$ are two Sasaki-extremal structures.  Then there is
an $g\in\Fol(\mathscr{F}_\xi,J)$ so that $g^*\omega^T_1 =\omega_0^T$.
\end{corint}

We also have sub-slope inequality for the relative K-energy $\cM^V$.
\begin{corint}\label{cor:Mab-rel-bound}
Suppose $\phi_0,\phi_1 \in\cH^G$, then the following inequality holds
\[ \cM^V(\phi_1) -\cM^V(\phi_0) \geq -d(\phi_0,\phi_1)\bigl(\Cal^G_{M,\xi}(\phi_0)\bigr)^{\frac{1}{2}}, \]
where $d$ is the distance function of the Mabuchi metric on $\cH$.
\end{corint}
Here we use a relative version of the Calabi functional $\Cal^G_{M,\xi}(\phi) =\int_M \bigl( S_\phi^G \bigr)^2 \, d\mu_\phi$,
where $S_\phi^G$ is the reduced scalar curvature, which is zero precisely when $\phi\in\cH^G$ gives an extremal
structure.  Thus a Sasaki-extremal structure is a minimum of the relative K-energy $\cM^V$.
Corollary~\ref{cor:Mab-bound} and Corollary~\ref{cor:Mab-rel-bound} provide interesting obstructions
to the existence of cscS and Sasaki-extremal structures.  The existence of a cscS (respectively
Sasaki-extremal) structure requires that the K-energy (respectively relative K-energy) is bounded below and that
it achieves its minimum.

Both uniqueness results Corollary~\ref{corint:unique-cscS} and Corollary~\ref{corint:unique-Sasak-ext}
are a consequence of the convexity of $\cM$ and $\cM^V$ along weak geodesics.  But since these
functionals are not known to be strictly convex an additional deformation technique is needed to
prove these results.  This is done in \S~\ref{sec:unique}.  This involves deforming $\cM$ by adding
a strictly convex functional $\cF^{\mu}$
\[  \cM^{t\mu} =\cM +t\cF^{\mu} \]
for small $t$.  Using an implicit function theorem argument we prove that $\cM^{t\mu}$ has a path of
critical points $\phi_t \in\cH$ for $t\in [0,\epsilon)$ for a particular potential $\phi_0$ in the orbit
of a cscS metric.  This is proved in Proposition~\ref{prop:csc-def}.  A bifurcation technique, due to
X. Chen, M. P\u{a}un, and Y. Zeng~\cite{CPZ15}, must be used since the differential of the map used has a kernel.
Uniqueness then follows from the strict convexity of $\cM^{t\mu}$.

The corresponding deformation result for the Sasaki-extremal case is given in Proposition~\ref{prop:S-E-def}.
The proof is similar, but more technicalities involving automorphism groups and the relative K-energy
$\cM^V$ need to be addressed.  Uniqueness again follows from the strict convexity of a deformed functional
\[ \cM^{V,t\mu}= \cM^V  +t\cF^{\mu} \]
where here the functional is defined on potentials $\cH^G$.

Uniqueness can be slightly generalized.  We say that a manifold with a transversely holomorphic foliation with
one dimensional leaves $(M,\mathscr{F},J)$ is of Sasakian type if it admits a Sasakian structure with $(\mathscr{F},J)$, with
its transversely holomorphic structure $J$ as its Reeb foliation.  The next result shows that for such a foliated manifold
$(M,\mathscr{F},J)$ a compatible Sasaki-extremal structure is unique up to homotheties and varying the contact form
by harmonic representatives of $H_b^1(M,\R) =H^1(M,\R)$.

\begin{corint}\label{cor:unique-gen}
Suppose $(\eta_0,\xi_0,\omega^T_0),(\eta_1,\xi_1,\omega^T_1)$ are two Sasaki-extremal structures
compatible with $(M,\mathscr{F},J)$, then there is a $g\in\Fol(M,\mathscr{F},J)$ and an $a>0$ so that $g^*a\omega^T_1 =\omega_0^T$
and $g_* \xi_0 =a^{-1} \xi_1$.
\end{corint}

More precisely, in the Corollary we have
\[g^* (a\eta_1 ,a^{-1}\xi_1 ,a\omega^T_1 ) =(\hat{\eta}_0, \xi_0 ,\omega_0^T),\]
where $\hat{\eta}_0 =\eta_0 +\alpha$ and $\alpha$ is a harmonic representative of $H_b^1(M,\R)$.
This result solves the uniqueness problem of Sasaki-extremal structures for a fixed transversely holomorphic foliation,
since varying the contact form $\eta$ with a harmonic representative of an element of
$H_b^1(M,\R) =H^1(M,\R)$, does not effect the scalar curvature or the transversal metric.

The above techniques give results in the \emph{$\alpha$-twisted} setting.  This approach has shown promise
in tackling problems in K\"{a}hler geometry~\cite{Che15}, so it is of interest in Sasakian geometry.
Let $\alpha$ be a closed, basic, positive
$(1,1)$-form.  The $\alpha$-twisted transversal scalar curvature of $(\eta,\xi,\Phi,g)$ is
\begin{equation}\label{eq:twist-sc}
S_g^T -\tr_{\omega^T}\alpha.
\end{equation}
A Sasakian metric is twisted cscS if (\ref{eq:twist-sc}) is a constant, and a Sasakian metric is twisted
Sasaki-extremal if (\ref{eq:twist-sc}) is the potential of a transversely holomorphic vector field.

\begin{thmint}
Suppose that $(\eta_0,\xi,\omega^T_0),(\eta_1,\xi,\omega^T_1) \in\cS(\xi,J)$ are two twisted constant scalar curvature  structures. Then $\omega^T_0 =\omega_1^T$.
\end{thmint}
We are able to prove a partial uniqueness result for twisted Sasaki-extremal structures.

It has been pointed out to the author that some of the same results are in the article of Xishen Jin and Xi Zhang~\cite{JinZha15},
though this work was done entirely independently.

\section{Saakian geometry and transversal space}\label{sec:Sasak-trans}

\subsection{Sasakian manifolds}

We review some of the properties of Sasakian manifolds that we will use.  See the monograph~\cite{BoGa08} for
details.

\begin{defn}
A Riemannian manifold $(M,g)$ is a \emph{Sasaki manifold}, or has a compatible Sasaki structure, if the metric cone
$(C(M),\ol{g})=(\R_{>0} \times M, dr^2 +r^2 g)$ is K\"{a}hler with respect to some complex structure $I$, where $r$ is the
usual coordinate on $\R_{>0}$.
\end{defn}
Thus $\dim M$ is odd and denoted $n=2m+1$, while $C(M)$ is a complex manifold with $\dim_{\C} C(M) =m+1$.

 We will identify $M$ with the $\{1\}\times M\subset C(M)$.  Let $r\partial_r$ be the Euler vector field on $C(M)$.
Using the warped product formulae for the cone metric $\ol{g}$~\cite{ONeil83} it is easy check that $r\partial_r$ is real holomorphic, $\xi$ is Killing with respect to both $g$ and $\ol{g}$, and furthermore the orbits of $\xi$ are geodesics on $(M,g)$.
Define $\eta =\frac{1}{r^2}\xi\contr\ol{g}$, then we have
\begin{equation}
\eta =-\frac{I^* dr}{r} =d^c \log r,
\end{equation}
where $d^c =\sqrt{-1}(\ol{\partial} -\partial)$.  If $\omega$ is the K\"{a}hler form of $\ol{g}$, then
\begin{equation}\label{eq:Kaehler-pot1}
\omega =\frac{1}{2}d(r^2 \eta)=\frac{1}{4}dd^c(r^2).
\end{equation}
From (\ref{eq:Kaehler-pot1}) we have
\begin{equation}\label{eq:Kaehler-pot2}
\omega=rdr\wedge\eta +\frac{1}{2}r^2 d\eta.
\end{equation}

Then (\ref{eq:Kaehler-pot2}) implies that
$\eta$ is a contact form with Reeb vector field $\xi$, since $\eta(\xi)=1$ and $\mathcal{L}_{\xi} \eta =0$.
Let $D\subset TM$ be the contact distribution which is defined by
\begin{equation}
D_x =\ker\eta_x
\end{equation}
for $x\in M$.  Furthermore, if we restrict the almost complex structure to $D$, $J:=I|_D$, then $(D,J)$ is a strictly pseudoconvex CR structure on $M$.  We have a splitting of the tangent bundle $TM$
\begin{equation}
TM =D\oplus L_{\xi},
\end{equation}
where $L_{\xi}$ is the trivial subbundle generated by $\xi$.  It will be convenient to define a tensor $\Phi\in\End(TM)$ by
$\Phi|_D =J$ and $\Phi(\xi) =0$.  Then
\begin{equation}\label{eq:comp-tens}
\Phi^2 =-\mathbb{1} +\eta\otimes\xi.
\end{equation}
Since $\xi$ is Killing, we have
\begin{equation}
d\eta (X,Y) =2 g(\Phi(X),Y), \quad\text{where }X,Y\in\Gamma(TM),
\end{equation}
and $\Phi(X) =\nabla_X \xi$, where $\nabla$ is the Levi-Civita connection of $g$.  Making use of (\ref{eq:comp-tens}) we see
that
\[ g(\Phi X,\Phi Y) =g(X,Y) -\eta(X)\eta(Y), \]
and one can express the metric by
\begin{equation}\label{eq:metric}
g(X,Y)=\frac{1}{2}(d\eta)(X,\Phi Y) +\eta(X)\eta(Y).
\end{equation}
We will denote a Sasaki structure on $M$ by $(g,\eta,\xi,\Phi)$.

\subsection{Transversal holomorphic structure}

We now describe a transverse K\"{a}hler structure on $\mathscr{F}_{\xi}$.
The vector field $\xi -\sqrt{-1}I\xi =\xi +\sqrt{-1}r\partial_r$ is holomorphic on $C(M)$.  If we denote by $\tilde{\C}^*$ the
universal cover of $\C^*$, then $\xi +\sqrt{-1}r\partial_r$ induces a holomorphic action
of $\tilde{\C}^*$ on $C(M)$.  The orbits of $\tilde{\C}^*$ intersect $M\subset C(M)$ in the orbits of the Reeb foliation
generated by $\xi$.  We denote the Reeb foliation by $\mathscr{F}_\xi$.  This gives $\mathscr{F}_\xi$ a transversely holomorphic
structure.

The foliation $\mathscr{F}_{\xi}$ together with its transverse holomorphic structure is given by an open covering
$\{U_\alpha \}_{\alpha\in A}$ of $M$ by product neighborhoods.  That is, there are charts
\begin{equation}\label{eq:fol-prod}
\Psi_\alpha : U_\alpha \rightarrow W_\alpha \times (-\epsilon,\epsilon),
\end{equation}
with $W_\alpha \subset\C^m$, where $\Psi_\alpha(x) =(\phi_\alpha (x),\tau_\alpha)$ and the leaves are locally given by
$\phi_\alpha^{-1}(z)$ for $z\in W_\alpha$.  And we may assume that $\xi$ is mapped to $\del_t$ in the
coordinates $(Z_\alpha, t_\alpha)$ on $W_\alpha \times (-\epsilon,\epsilon)$.  When $U_\alpha \cap U_\beta \neq\emptyset$
the transition maps
\begin{equation}\label{eq:fol-prod-trans}
 \Psi_\beta \circ\Psi_\alpha^{-1} :\Psi_\alpha (U_\alpha \cap U_\beta) \rightarrow\Psi_\beta (U_\alpha \cap U_\beta)
\end{equation}
are given by $ \Psi_\beta \circ\Psi_\alpha^{-1}(z,t)=(\phi_\beta \circ\phi_\alpha^{-1}(z),t +\theta_{\beta\alpha}(z))$.
Since $\mathscr{F}_{\xi}$ is \emph{transversely holomorphic} the transitions
\begin{equation}\label{eq:fol-tran}
 \phi_{\beta\alpha} :=\phi_\beta \circ\phi_\alpha^{-1}: W_\alpha \cap W_\beta \rightarrow W_\alpha \cap W_\beta
\end{equation}
are biholomorphisms, and satisfy the cocyle condition $\phi_{\gamma\beta}\circ\phi_{\beta\alpha} =\phi_{\gamma\alpha}$
on $W_\alpha \cap W_\beta \cap W_\gamma$.
The transversal K\"{a}hler form $\omega^T$ induces a K\"{a}hler form $\omega_\alpha$ on $W_\alpha$
with $\phi_{\beta\alpha}^* \omega_\beta =\omega_\alpha$.

In working on the transversal space we work on the charts $\{W_\alpha\}_{\alpha\in A}$ with their K\"{a}hler structure
invariant under the transitions $\phi_{\beta\alpha}$.  But it will also be useful to consider \emph{basic} functions and tensors.
If we define $\nu(\mathscr{F}_\xi) =TM/{L_\xi}$ to be the normal bundle to the leaves, then we can generalize the above concept.
\begin{defn}
A tensor $\Psi\in\Gamma\bigl((\nu(\mathscr{F}_\xi)^*)^{\otimes p} \bigotimes\nu(\mathscr{F}_\xi)^{\otimes q}\bigr)$ is \emph{basic} if
$\mathcal{L}_V \Psi =0$ for any vector field $V\in\Gamma(L_\xi)$.
\end{defn}
It is sufficient to check this for $V=\xi$.
Then $g^T$ and $\omega^T$ are such tensors on $\nu(\mathscr{F}_\xi)$.  We will also make use of the bundle isomorphism
$\pi:D \rightarrow\nu(\mathscr{F}_\xi)$, which induces an almost complex structure $\ol{J}$ on $\nu(\mathscr{F}_\xi)$ so that
$(D,J)\cong(\nu(\mathscr{F}_\xi),\ol{J})$ as complex vector bundles.  Clearly, $\ol{J}$ is basic and is mapped by the
foliation charts $\phi_\alpha$ to the complex structure on $W_\alpha$.  In the sequel we will denote the almost
complex structure on $\nu(\mathscr{F}_\xi)$ by $J$ and denote the Reeb foliation with its transversal holomorphic structure
by $(\mathscr{F}_\xi,J)$.

Smooth basic functions will be denoted by $C^\infty_b(M)$.  The basic exterior r-forms are denoted $\Omega^r_b$, and
split into types
\[\Omega^r_b =\bigoplus_{p+q =r} \Omega_b^{p,q}. \]

To work on the K\"{a}hler leaf space we define the Levi-Civita connection of $g^T$ by
\begin{equation}
\nabla^T_X Y =\begin{cases}
\pi_\xi(\nabla_X Y) & \text{ if }X, Y\text{ are smooth sections of }D, \\
\pi_\xi([V,Y]) & \text{ if } X=V\text{ is a smooth section of }L_\xi,
\end{cases}
\end{equation}
where $\pi_\xi :TM \rightarrow D$ is the orthogonal projection onto $D$.  Then $\nabla^T$ is the unique torsion free connection
on $D\cong\nu(\mathscr{F}_\xi)$ so that $\nabla^T g^T=0$.  Then for $X,Y\in\Gamma(TM)$ and $Z\in\Gamma(D)$ we have the
curvature of the transverse K\"{a}hler structure
\begin{equation}
R^T(X,Y)Z =\nabla^T_X \nabla^T_Y Z -\nabla^T_Y \nabla^T_X Z -\nabla^T_{[X,Y]} Z,
\end{equation}
and similarly we have the transverse Ricci curvature $\Ric^T$ and scalar curvature $S^T$.

The following follows from O'Neill tensor computations for a Riemannian submersion.  See~\cite{ONe66} and~\cite[Ch. 9]{Be87}.
\begin{prop}\label{prop:Sasaki-Ric}
Let $(M,g,\eta,\xi,\Phi)$ be a Sasaki manifold of dimension $n=2m+1$, then
\begin{thmlist}
\item  $\Ric_g (X,\xi) =2m\eta(X),\quad\text{for }X\in\Gamma(TM)$,\label{eq:submer-Ric-Reeb}
\item  $\Ric^T (X,Y) =\Ric_g (X,Y) +2g^T(X,Y),\quad\text{for }X,Y\in\Gamma(D),$
\item  $S^T =S_g +m.$\label{eq:submer-scal}
\end{thmlist}
\end{prop}
Here we define $S_g$ (respectively $S^T$) to be $1/2$ the trace of $\Ric_g$ with respect to $g$ (respectively 1/2 the trace of
$\Ric^T$ with respect to $g^T$) to simplify notation later on.

We define $\mathcal{S}(\xi,J)$ to be the set of Sasakian structures with Reeb vector field $\xi$ and
with the holomorphic structure $J$ on the Reeb foliation $\mathscr{F}_\xi$.
In other words, the set of Sasakian structures inducing the same complex normal bundle $(\nu(\mathscr{F}_\xi),J)$.
This is the set of $(\tilde{g},\tilde{\eta},\xi,\tilde{\Phi})\in\mathcal{S}(\xi)$ such that the following diagram
commutes
\begin{equation}\label{eq:trans-CD}
\begin{CD}
TM @>{\tilde{\Phi}}>> TM \\
@VVV            @VVV \\
\nu(\mathscr{F}_\xi) @>{J}>> \nu(\mathscr{F}_\xi).
\end{CD}
\end{equation}

The next lemma describes $\mathcal{S}(\xi,J)$ in detail.  Define
\[ \cH_{\omega^T} =\{ \phi\in C^\infty_b (M)\ |\ (\omega^T +dd^c \phi)^m \wedge\eta >0\ \} \]

\begin{lem}[\cite{BoGa08,BoGaSi08}]\label{lem:trans-def}
The space $\mathcal{S}(\xi,J)$ of all Sasaki structures with Reeb vector field $\xi$ and transverse holomorphic
structure $J$  is an affine space modeled on $\cH /\R \times C^\infty_b (M)/\R \times H^1(M,\R)$.
If $(g,\eta,\xi,\Phi)\in\mathcal{S}(\xi,J)$ is a fixed Sasaki structure then another structure
$(\tilde{g},\tilde{\eta},\tilde{\xi},\tilde{\Phi})\in\mathcal{S}(\xi,J)$ is determined by real basic functions
$\phi$ and $\psi$ and an harmonic, with respect to $g$, 1-form $\alpha$ such that
\begin{equation}\label{eq:trans-def}
\begin{split}
\tilde{\eta} & =\eta + 2d^c \phi +d\psi +\alpha, \\
\tilde{\Phi} & =\Phi -\xi\otimes\tilde{\eta}\circ\Phi,\\
\tilde{g}    & =\frac{1}{2}d\tilde{\eta}\circ(\mathbb{1}\otimes\tilde{\Phi}) +\tilde{\eta}\otimes\tilde{\eta},
\end{split}
\end{equation}
and the transversal K\"{a}hler form becomes $\tilde{\omega}^T =\omega^T +dd^c \phi$.
\end{lem}
\begin{proof}
We give only a sketch.  See~\cite{BoGa08} for details.  The 1-form $\gamma=\tilde{\eta}-\eta$ is basic, and since
$d\gamma\in\Gamma(\Lambda_b^{1,1})$ and $\gamma$ is real, $d^c d\gamma =0$.  And we have the Hodge decomposition
\begin{equation}
\gamma =d^c \phi + d\psi +\alpha,
\end{equation}
with respect to the transversal K\"{a}hler metric $g^T$, where $\alpha\in\mathcal{H}^1_{g^T}$ is harmonic.
But note that $\mathcal{H}_{\R,g^T}^1 =\mathcal{H}_{\R,g}^1$, where the latter is the space of real harmonic 1-forms on $(M,g)$.
This is because a $\beta\in\Gamma(\Lambda^1(M))$ satisfying $d\beta =0$ and $\mathcal{L}_\xi \beta=0$ must be basic.
\end{proof}

It is easy to check that the parameter $\psi$ in (\ref{eq:trans-def}) changes the structure only by a gauge transformation along
the leaves.  That is, if $\psi\in C^\infty_b(M)$, then $\exp(\psi\xi)^*\eta =\eta +d\psi$.
Altering by a harmonic form likewise does not effect the transversal metric, so it will not be of much interest.

Given $\phi\in\cH_{\omega^T}$ we define the \emph{transversal K\"{a}hler deformation} of $(g,\eta,\xi,\Phi)$
to be the Sasakian structure given in the lemma, which we will denote
$(g_\phi,\eta_\phi,\xi,\Phi_\phi)$.  We will denote by $\omega^T_\phi =\omega^T +dd^c \phi$ the transversal K\"{a}hler
form of the deformed structure.

This article is concerned with constant scalar curvature Sasakian structures (cscS) and more generally
Sasaki-extremal structures.  By Proposition~\ref{prop:Sasaki-Ric} the scalar curvature $S_g$ and the
scalar curvature of the transversal structure $S_g^T$ differ by a constant, both these conditions are
given by the transversal K\"{a}hler structure.

Let $\fol(M,\sF,J)$ be the space of vector fields preserving the Reeb foliation along with its transversal
holomorphic structure.  The corresponding group is denoted by $\Fol(M,\sF,J)$.  This is an infinite dimensional
group, since any vector field tangent to the leaves is in $\fol(M,\sF,J)$.  So we define
$\hol^T(\xi,J)$ to be the image of
\begin{equation}
\begin{array}{rcl}
\fol(M,\mathscr{F}_\xi,J) & \overset{\pi}{\longrightarrow} & \Gamma(\nu(\mathscr{F}_\xi)) \\
X & \mapsto & \ol{X}
\end{array}
\end{equation}
which is a finite dimensional complex Lie algebra.  We will use $\hol^T(\xi,J)$ to denote both transversally
holomorphic $(1,0)$ vector fields, or transversally real holomorphic vector fields
depending on the context.

Given a basic $\phi\in C^\infty_b(M,\C)$, we define $\del_g^{\#} \phi$ to be the $(1,0)$ component of the gradient, that is
\begin{equation}
g(\del_g^{\#} \phi,\cdot) =\ol{\del}\phi.
\end{equation}
In order for $\del_g^{\#} \phi \in\hol^T(\xi,J)$, transversely holomorphic,
we need in addition $\ol{\del}_b \del_g^{\#} \phi =0$.
This is equivalent to the fourth-order transversally elliptic equation
\begin{equation}
\Li_g \phi := (\ol{\partial}\partial_g^{\#})^*\ol{\partial}\partial_g^{\#}\phi.
\end{equation}
We have
\begin{equation}\label{eq:hol-oper}
\Li_g \phi =\frac{1}{4}\Delta_b^2 \phi +\frac{1}{2}(\rho^T,dd^c \phi) +(\del S^T)\contr\del_g^{\#}\phi.
\end{equation}
We define the space of \emph{holomorphy potentials} to be  $\mathscr{H}_g :=\ker\Li_g$.

We define the \emph{Calabi functional}
\[\Cal_{M,\xi} :\mathcal{S}(\xi,J) \rightarrow\R \]
\begin{equation}\label{eq:Cal-func}
\Cal_{M,\xi}(g) := \int_M (S_{g} -\ol{S})^2 \, d\mu_g,
\end{equation}
where $d\mu_g =(\omega^T)^m \wedge\eta$.
\begin{defn}
A Sasakian structure $(g,\eta,\xi,\Phi)$ is \emph{Sasaki-extremal} if it is a critical point of the Calabi functional.
Equivalently, $\del^{\#} S_g$ is transversally holomorphic.
\end{defn}
See~\cite{BoGaSi08} for the proof that the two conditions are equivalent.

A function $u:W\rightarrow\R\cup\{-\infty\}$ on an open set $W\subset\C^m$ is \emph{quasi-plurisubharmonic}
if it can locally be written as the sum of a plurisubharmonic function and a smooth function.  Recall that a
function $u$ on $W$ is plurisubharmonic if
\begin{thmlist}

\item  $u$ is upper semicontinuous;

\item  for every complex line $L\subset\C^m$, $u|_{L\cap L}$ is subharmonic on $L\cap W$.
\end{thmlist}

Let $\theta$ be any closed basic $(1,1)$-form.
\begin{defn}
A function $u:M\rightarrow\R\cup\{-\infty\}$ is said to be transversally $\theta$-plurisubharmonic ($\theta$-psh) if u is invariant
under the Reeb flow, is upper semi-continuous, in each foliation chart $W_\alpha$ $u$ is quasi-plurisubharmonic and
\[ \theta +dd^c u\geq 0, \]
as a $(1,1)$-current.

The set of $\theta$-psh functions on $M$ is denoted $\PSH(M,\theta)$.
\end{defn}

Because a function $u\in\PSH(M,\theta)$ is defined to be $\theta$-psh in each holomorphic foliation chart $W_\alpha$
most of the familiar properties translate into this situation.  For example $u\in L^1(d\mu_{\eta})$ where
$d\mu_{\eta} =(\omega^T)^m \wedge\eta$.  See~\cite{Dem12}.

The following approximation result will be useful.
\begin{prop}\label{prop:psh-aprox}
Suppose $\theta$ is a positive basic $(1,1)$-form on $M$.
Let $\phi\in\PSH(M,\theta)\cap C^0(M)$, then there exists a sequence $\phi_j \in\PSH(M,\theta)$ decreasing to $\phi$.
\end{prop}
The proof will mostly follow from the following.

\begin{thm}[\cite{BloKo07}]
Let $X$ be a complex manifold with a positive hermitian form $\omega$ and $\gamma$ a continuous $(1,1)$-form on $M$.
Let $\phi\in\PSH(X,\gamma)$ be locally bounded.  Then for any relatively compact open $X'\subset X$ we have a decreasing sequence
$\varepsilon_j \searrow 0$ and a sequence $\phi_j \in\PSH(X', \gamma+\varepsilon_j \omega)\cap C^\infty(X')$ decreasing to $\phi$.
\end{thm}

\begin{proof}[Proof of Proposition]
Let $X=C(M)$ with $M=\{r=1\}\subset X$.  And let $X'\subset X$ be $X'=\{(r,x)\in C(M)\ |\ 1-\epsilon <r<1+\epsilon \}$.

For $\phi\in\PSH(M,\theta)\cap C^0(M)$ consider $\phi$ to be a function on $X$ via the projection $p: X\rightarrow M$,
$p(r,x)=x$, and similarly consider $\theta$ as a $(1,1)$-form on $X$.  Define
\[ \omega = \frac{dr}{r}\wedge\eta +\theta \]
a positive hermitian form on $X$.  By the theorem there exists
$\psi_j \in\PSH(X', \theta +\varepsilon_j \omega)\cap C^\infty(X')$ with $\psi_j \searrow\phi$.
Let $T\subset\Aut(M,\eta,\xi,g)$ be the torus generated by $\xi$.  By averaging by $T$ we may assume that $\psi_j$ are
invariant by $\xi$.

Suppose that $\psi\in C^\infty(X')$ is invariant under $\xi$.  Routine calculation shows that the complex hessian at a point of
$M\subset X'$ for $X,Y\in\Gamma(D)$ basic is
\begin{equation}\label{eq:hess-cr}
 dd^c \psi (X,Y) = d_b d^c_b \psi(X,Y) + 2d\psi(\del_r)\omega^T(X,Y).
\end{equation}

On $X'$ we have
\begin{equation}\label{eq:hess-pos}
 \varepsilon_j \frac{dr}{r}\wedge\eta +(1+\varepsilon_j)\theta +dd^c \psi_j \geq 0.
\end{equation}
Introduce the coordinate $t=\log r$ so $\del_t =r\del_r$.  Then substituting $(\del_t,\xi)$ into the above inequality
gives $\varepsilon_j +\del_t^2 \psi_j \geq 0$.  Then $\hat{\psi}_j =\psi_j +\varepsilon_j t^2$ is convex with respect to
$t$ and converges uniformly to $\phi$ on $[-\delta,\delta]\times M =\{(t,x)\in X' \ |\ -\delta<t<\delta \}\subset X'$.
By convexity we have
\[ \frac{\hat{\psi}_j(0,x) -\hat{\psi}_j(-\delta,x)}{\delta} \leq\del_t \psi_j (0,x)\leq
\frac{\hat{\psi}_j(\delta,x) -\hat{\psi}_j(0,x)}{\delta}\]
Or
\[ \frac{\psi_j(0,x) -\psi_j(-\delta,x)-\varepsilon_j \delta^2}{\delta}\leq \del_t \psi_j(0,x)\leq
\frac{\psi_j(\delta,x) -\phi_j(o,x) +\epsilon_j \delta^2}{\delta}. \]
Thus $\del_t \psi_j \rightarrow 0$ uniformly on $M\subset X'$.

From (\ref{eq:hess-cr}) and (\ref{eq:hess-pos}) on $M$ we have
\[ (1+\varepsilon_j)\theta +d_b d^c_b \psi_j +2d\psi_j(\del_t)\omega^T \geq 0. \]
After possibly passing to a subsequence of $\psi_j$ there are constants $\hat{\varepsilon}_j\searrow 0$ with
\[ (1+\hat{\varepsilon}_j)\theta + d_b d^c_b \psi_j \geq 0.\]
Without loss of generality we may assume that $\phi\leq -1$.  Then we have $\lambda_j =1+\hat{\varepsilon}_j$,
$\lambda_j \searrow 1$, and negative $\psi_j \in\PSH(M,\lambda_j \theta)$.  Then
$\phi_j :=\psi_j /\lambda_j \in\PSH(M,\theta)\cap C^\infty(M)$ is a sequence decreasing to $\phi$.
\end{proof}

\subsection{Monge-Amp\`{e}re operator}

We will define a transversal version of the Monge-Amp\`{e}re operator on transversally quasi-plurisubharmonic
functions on $M$.  This will be needed to define weak geodesics in the space of Sasakian structures.  It will
also be needed to define necessary energy functionals on weak structures, in particular the Monge-Amp\`{e}re
energy and the Mabuchi K-energy.

We can define a \emph{transversal current} $T$ on the foliation $\mathscr{F}_{\xi}$ to be a collection
$\{W_\alpha,T_\alpha\}_{\alpha\in A}$ so that
$\phi_{\beta\alpha*} T_{\alpha}|_{W_\alpha \cap W_{\beta}}= T_{\beta}|_{W_\alpha \cap W_{\beta}}$.
Since that transition maps (\ref{eq:fol-tran}) are holomorphic, we define $T$ to have bidegree $(p,q)$ if
each $T_\alpha$ has bidegree $(p,q)$, $T_\alpha \in{\cD'}^{p,q}(W_\alpha)$.  Similarly, we define the notions of
a closed transversal current, respectively a positive current, to be transversal currents with each $(W_\alpha, T_\alpha)$ closed,
respectively positive.

If $\theta$ is any basic closed $(1,1)$-form, and $u\in\PSH(M,\theta)$, then $\theta +dd^c u$ is a closed positive
$(1,1)$ transversal current.  Suppose that $u\in\PSH(M,\theta)\cap L^\infty$ and $T$ is a closed positive transversal current
of bidegree $(p,p)$.  One can employ the Bedford-Taylor~\cite{BeTa76} construction to define
the closed, positive, degree $(p+1,p+1)$ transversal current
\[ (\theta +dd^c u)\wedge T. \]
This is of course defined in each chart $W_\alpha$, as follows.  If $\theta=dd^c w$, then
$dd^c(w+u)\wedge T_\alpha := dd^c\bigl((w+u)T_\alpha \bigr)$.  One can check that this is independent of $w$.

Given $u_1,\ldots, u_m \in\PSH(M,\theta)\cap L^\infty$ by applying this definition inductively, we get
\[ (\theta +dd^c u_1)\wedge\cdots\wedge (\theta+dd^c u_m), \]
a positive, bidegree $(m,m)$ current defined in each $W_\alpha$.  It therefore defines a Radon measure in each $W_\alpha$,
invariant under the transitions (\ref{eq:fol-tran}).

We define the \emph{Monge-Amp\`{e}re operator} as follows.  We define the measure, denoted
\begin{equation}\label{eq:def-MA-meas}
(\theta +dd^c u_1)\wedge\cdots\wedge (\theta+dd^c u_m)\wedge\eta,
\end{equation}
to be the product measure on $U_\alpha \cong W_\alpha \times (-\epsilon,\epsilon)$ given by the chart (\ref{eq:fol-prod}).
It is easy to see that this is invariant of the transition maps (\ref{eq:fol-prod-trans}).
This is easy to see using Fubini's theorem, and the invariance of $(\theta +dd^c u_1)\wedge\cdots\wedge (\theta+dd^c u_m)$
under the holomorphic transitions (\ref{eq:fol-tran}).

The Monge-Amp\`{e}re operator is defined as locally a product, so many of the usual properties of the usual
Monge-Amp\`{e}re operator on complex manifolds hold.  The most important will be weak convergence under several cases
of convergence of function.
\begin{thm}\label{thm:weak-con}
Let $u_0^j$ be a sequence of bounded transversally quasi-psh functions, and sequences
$u_1^j,\ldots, u_m^j \in \PSH(M,\theta)\cap L^\infty$.  Then

\[ u_0^j (\theta +dd^c u_1^j)\wedge\cdots\wedge (\theta+dd^c u_m^j)\wedge\eta \]

converges weakly to

\[  u_0 (\theta +dd^c u_1)\wedge\cdots\wedge (\theta+dd^c u_m)\wedge\eta, \]

where $u_0$ is transversal quasi-psh and bounded and $u_1,\ldots, u_m \in \PSH(M,\theta)\cap L^\infty$, when
the convergence $u_k^j \rightarrow u_k$ for each k is one of the following.
\begin{itemize}
\item  $u_k^j$ decreases pointwise to $u_k$.

\item  $u_k^j$ increases to $u_k$ a.e. with respect to Lebesgue measure.

\item  $u_k^j$ converges to $u_k$ uniformly on $M$.

\end{itemize}
\end{thm}

We note that if $T$ is a closed positive transversal current of bidegree $(m-1,m-1)$ and $u,v$ are bounded
transversal quasi-psh then $du \wedge d^c v \wedge T \wedge\eta$ can be defined.  We may suppose $u\geq 0$ and define
\[  du \wedge d^c u \wedge T \wedge\eta := \frac{1}{2} dd^c u^c \wedge T -udd^c u\wedge T \wedge\eta,\]
and the general case can be defined by polarization.
In particular, $du \wedge d^c u \wedge  T \wedge\eta \geq 0$ and we have the analogous convergence as in Theorem
~\ref{thm:weak-con}.

Integration by parts formulae will be useful.
\begin{prop}\label{prop:int-parts}
Suppose that $\theta$ is a positive, closed, basic $(1,1)$-form on $M$.  Let $v,w$ each be differences of continuous
transversally  quasi-psh functions, and let $u_1,\ldots,u_{m-1} \in\PSH(M,\theta)\cap C^0(M)$.  Then
\begin{multline}
\int_M vdd^c w \wedge(\theta+dd^c u_1)\wedge\cdots\wedge (\theta +dd^c u_{m-1})\wedge\eta \\
  =\int_M wdd^c v \wedge(\theta+dd^c u_1)\wedge\cdots\wedge (\theta +dd^c u_{m-1})\wedge\eta \\
  =-\int_M dv\wedge d^c w \wedge(\theta+dd^c u_1)\wedge\cdots\wedge (\theta +dd^c u_{m-1})\wedge\eta.
\end{multline}
\end{prop}
\begin{proof}
By assumption $v=q-r, w=s-t$ with $q,r,s,t$ quasi-psh.  By Proposition~\ref{prop:psh-aprox} $q,r,s,t,u_1,\ldots,u_{m-1}$
can be approximated by decreasing sequences of smooth transversally quasi-psh functions.  The above equations hold
for these approximations by Stoke's theorem.  Then the result follows from Theorem~\ref{thm:weak-con}.
\end{proof}

\subsection{Weak geodesics}

We are primarily concerned with applications of the Monge-Amp\`{e}re operator to weak geodesics in $\cH$.
We describe the weak Riemannian structure on $\cH$.  Given $\phi\in\cH$ and $\psi_1,\psi_2 \in T_\phi \cH\cong C_b^\infty(M)$

\[ \langle\psi_1 ,\psi_2 \rangle_\phi :=\int_M \psi_1 \psi_2 \, d\mu_\phi ,\]
where $d\mu_\phi =(\omega_\phi^T)^m \wedge\eta_\phi$.

There is a torsion free connection compatible with this metric.  Given a smooth path $\{\phi_t | a\leq t\leq b \}$ in $\cH$
a vector field along $\{ \phi_t \}$ can be identified with a smooth path $\{ \psi_t | a\leq t\leq b \}\in C_b^\infty([a,b]\times M)$.
The covariant derivative, in transversal holomorphic coordinates (\ref{eq:fol-prod}), is

\[ \frac{D}{\del t} := \del_t -\frac{1}{2\sqrt{-1}}\sum\bigl(\omega_{\phi}^T \bigr)^{\alpha\ol{\beta}}\bigl(\dot{\phi}_{\alpha}\del_{\ol{\beta}}+\dot{\phi}_{\ol{\beta}}\del_{\alpha} \bigr).\]
The geodesic equation is then
\begin{equation}\label{eq:geo}
\ddot{\phi} =\frac{1}{2}|d\dot{\phi}|^2_{\omega^T_{\phi}}
\end{equation}
for a smooth path $\{\phi_t | a\leq t\leq b \}\subset\cH$.

We define the \emph{Monge-Amp\`{e}re energy} to be the potential $\cE:\cH\rightarrow\R$ with derivative
\[ d\cE|_{\phi}(\psi)=\int_M \psi\, d\mu_\phi, \]
which is easily seen to be closed.  So we may define
\begin{equation}\label{eq:M-A-pre}
\cE(\phi)=\int_0^1 \int_M \dot{\phi}_t \, d\mu_{\phi_t} dt,
\end{equation}
where $\{\phi_t \ |\ 0\leq t\leq 1 \}$ is a smooth path in $\cH$ with $\phi_0 =0,\phi_1 =\phi$.

Define
\begin{equation}\label{eq:nor-Kah-pot}
\tilde{\cH} =\{ \phi\in\cH\ |\ \cE(\phi)=0 \}\subset\cH.
\end{equation}

The map
\[\begin{array}{ccl}
\cH & \cong & \tilde{\cH}\times\R \\
\phi & \leftrightarrow & (\phi -\frac{\cE(\phi)}{\Vol(M)}, \frac{\cE(\phi)}{\Vol(M)})
\end{array} \]
is an isometry.  And $\tilde{\cH}$ is geodesically convex, that is a geodesic $\{ \phi_t \ |\ a\leq t\leq b \}$ with
$\phi_a,\phi_b \in\tilde{\cH}$ is contained in $\tilde{cH}$.

If $\cK$ denotes the space of Sasakian structures associated to $\cH$, then we have an isomorphism
\[\begin{array}{ccl}
\tilde{\cH} & \cong & \tilde{\cK} \\
\phi & \leftrightarrow & (\eta_\phi ,\xi,\Phi_\phi, g_\phi)
\end{array} \]
A geodesic in $\cK$ is defined to be a geodesic in $\tilde{\cH}$.

Let $A=\{ \tau\in\C\ |\ 1\leq |\tau|\leq e \}$, then $N:= M\times A$ is a manifold with boundary, with a transversely
holomorphic foliation.  The foliation charts are as in (\ref{eq:fol-prod}).  If $V\subset A$, then the charts are
\begin{equation}\label{eq:fol-prod-N}
 \Phi_\alpha :U_\alpha \times V \rightarrow W_\alpha \times V\times(-\epsilon,\epsilon)
\end{equation}
with $W_\alpha \times V$ giving the local holomorphic leaf space.

A path $\phi\in C_b^\infty([0,1]\times M)$ corresponds to an $S^1$-invariant function $\Phi_\tau$ on
$N$ under $\tau =e^t$.  If $\{\phi_t \}$ is a smooth path in $\cH$ then a routine calculation shows that

\[ (\pi^* \omega^T +dd^c \Phi)^{m+1} =\frac{(m+1)}{4}(\ddot{\phi}_t -\frac{1}{2}|d\dot{\phi}_t|^2_{\omega_{\phi_t}^T})
   (\omega^T_{\phi_t})^m \wedge\frac{d\tau\wedge d\ol{\tau}}{|\tau|^2}.\]
Thus a smooth geodesic between $\phi_0,\phi_1 \in\cH$ is given by $\Phi\in C_b^\infty(N)$ with
$\Phi_\tau =\phi_j, |\tau|=e^j, j=0,1$, and $\omega^T +dd^c \Phi_\tau >0$ for all $\tau\in A$ so that
\[ (\pi^* \omega^T +dd^c \Phi)^{m+1} =0.\]
We can define a \emph{weak geodesic} between $\phi_0,\phi_1 \in\cH$ as follows.
\begin{equation}\label{eq:weak-geo}
\begin{cases}
\Phi\in\PSH(\overset{\circ}{N},\pi^* \omega^T)\cap C^0(N) & \\
\Phi_\tau =\phi_j, |\tau|=e^j ,j=0,1 & \\
(\pi^* \omega^T +dd^c \Phi)^{m+1}\wedge\eta =0 &
\end{cases}
\end{equation}
We assume $\Phi\in C^0(N)$ merely because it is the weakest regularity we will consider.

The best regularity for a solution of the Dirichlet problem (\ref{eq:weak-geo}) is due to
P. Guan and Xi Zhang~\cite{GuaZha12}.  We define $C_{w}^{1,1}(N)$ to be the completion of $C^{\infty}_b(N)$ with
norm $\|\phi\|_w =\|\phi\|_{C^1} +\|dd^c \phi\|_{L^\infty}$.
\begin{thm}[\cite{GuaZha12}]\label{thm:weak-geo}
The Dirichlet problem (\ref{eq:weak-geo}) for $\phi_0,\phi_1 \in\cH$ has a unique solution
$\Phi\in C_w^{1,1}(N)$.
\end{thm}

We only have weak regularity to the non-elliptic problem (\ref{eq:weak-geo}), so for $\epsilon>0$
we define a path $\{\phi_t \ |\ 0\leq t\leq 1 \}$ in $\cH$ to be an \emph{$\epsilon$-geodesic} if
\begin{equation}
\bigl(\ddot{\phi}-\frac{1}{2}|d\dot{\phi}|^2_{\omega^T_{\phi}}\bigr)(\omega_{\phi}^T)^m =\epsilon(\omega^T)^m.
\end{equation}
An $\epsilon$-geodesic is necessarily a smooth path in $\cH$, since it is a solution to the transversely elliptic
problem
\[ (\pi^* \omega^T +dd^c \Phi)^{m+1} =\frac{\epsilon}{4}\bigl(\pi^* \omega^T +\frac{\sqrt{-1}d\tau\wedge d\ol{\tau}}{|\tau|^2}\bigr)^{m+1}. \]

It follows from the proof in~\cite{GuaZha12} and the maximal principle proved below that there are
smooth $\epsilon$-geodesics $\Phi^\epsilon$ monotonically decreasing in $\epsilon>0$ and
$\Phi^\epsilon \rightarrow \Phi$, the weak solution in Theorem~\ref{thm:weak-geo}, weakly in $C_{w}^{1,1}(N)$
as $\epsilon\rightarrow 0$.

\subsection{Maximal principle and uniqueness results}

We will prove a maximal principle for the Monge-Amp\`{e}re operator on $N$ and some uniqueness results.
First we give a version of Proposition~\ref{prop:int-parts} for $N$.
\begin{prop}\label{prop:int-parts2}
Let $\theta$ be a basic positive $(1,1)$-form on $N$.  Let $v,w$ each be differences of continuous
transversally  quasi-psh functions, and let $u_1,\ldots,u_{m} \in\PSH(N,\pi^*\theta)\cap C^0(N)$.
Then
\begin{multline}
\int_N vdd^c w \wedge(\pi^*\theta+dd^c u_1)\wedge\cdots\wedge (\pi^*\theta +dd^c u_{m})\wedge\eta \\
  =\int_N wdd^c v \wedge(\pi^*\theta+dd^c u_1)\wedge\cdots\wedge (\pi^*\theta +dd^c u_{m})\wedge\eta \\
  =-\int_N dv\wedge d^c w \wedge(\pi^*\theta+dd^c u_1)\wedge\cdots\wedge (\pi^*\theta +dd^c u_{m})\wedge\eta,
\end{multline}
provided one of $v,w,$ or $T=(\pi^*\theta+dd^c u_1)\wedge\cdots\wedge (\pi^*\theta +dd^c u_{m})$ has compact support
in $N\setminus\del N$.
\end{prop}
\begin{proof}
Suppose $v=q-r$ has compact support where $q,r$ are quasi-psh, and we may assume $q,r\leq -1$.
Let $f(\tau)=(\log |\tau|)^2 -\log|\tau|$ a strictly psh function on $A$ vanishing on $\del A$.
Choose $M>0$ large enough that $q,r > Mf$ outside of the interior of $U=\{x\in N\ |\ q(x)=r(x) \}$.
If we define $\tilde{q}=\max\{q,Mf\}$ and $\tilde{r}=\max\{r,Mf\}$.  Let $W\subset N\setminus\del N$ be
a relatively compact open neighborhood containing $\{\tilde{q}\geq Mf\}\cup\{\tilde{r}\geq Mf\}$.
Proposition~\ref{prop:psh-aprox} gives decreasing sequences of smooth transversely quasi-psh $q_i$ and $r_i$ on $W$
with $q_i \searrow q$ and $r_i \searrow r$.  The sequences can be chosen so that $q_i ,r_i <Mf$ near $\del W$.
Define $\tilde{q}_i=\max\{q_i,Mf\}$ and $\tilde{r}_i=\max\{r_i,Mf\}$, where we make take the regularized maximum
(See~\cite[I-5.18]{Dem12}) so that $\tilde{q}_i ,\tilde{r}_i$ are smooth.
Similarly, for each $k=1,\ldots,m$ choose a sequence $u^i_k \in\PSH(W,\pi^* \theta)\cap C^\infty_b(W)$
with $u^i_k \searrow u_k$.  And if $w=s-t$ with $s,t$ quasi-psh, we choose sequences $s_i,t_i$ of smooth
quasi-psh on $W$.

The integration by parts formula then holds with $v_i =\tilde{q}_i -\tilde{r}_i ,w_i =s_i -t_i$ and
$u^i_1,\ldots,u^i_m$ substituted by Stoke's theorem, since $v_i$ has compact support in $\ol{W}$.
Applying Theorem~\ref{thm:weak-con} finishes the proof.
\end{proof}

We prove weak maximal principle first.  Let $\theta$ be a basic positive $(1,1)$-form on $N$.
\begin{prop}\label{prop:weak-max}
Let $u,v\in\PSH(N,\pi^* \theta)\cap C^0(N)$ satisfy $u\leq v$ on $\del N$.  Then
\[ \int_{v<u}(\pi^* \theta +dd^c u)^{m+1}\wedge\eta \leq \int_{v<u}(\pi^* \theta +dd^c v)^{m+1}\wedge\eta. \]
\end{prop}
\begin{proof}
Let $\delta>0$ then $\Omega:=\{v <u -\delta \}\Subset N\setminus\del N$.  Define
$u^\epsilon :=\max\{u-\delta,v +\epsilon \}$ for small $\epsilon>0$, so $u^\epsilon =v+\epsilon$
in a neighborhood of $\del\Omega$.  We have
\[ \int_\Omega (\pi^* \theta +dd^c u^\epsilon)^{m+1}\wedge\eta = \int_\Omega (\pi^* \theta +dd^c v)^{m+1}\wedge\eta.\]
This is because
\begin{equation}\label{eq:MA-diff}
(\pi^* \theta +dd^c u^\epsilon)^{m+1}\wedge\eta -(\pi^* \theta +dd^c v)^{m+1}\wedge\eta
= dd^c(u^\epsilon -v)\wedge T\wedge\eta
\end{equation}
where $T=\sum_{j=0}^m (\pi^* \theta +dd^c u^\epsilon)^j \wedge(\pi^* \theta +dd^c v)^{m-j}$.
Then the integral of (\ref{eq:MA-diff}) is zero by Proposition~\ref{prop:int-parts2}.  Since $u^\epsilon$ decreases
to $u-\delta$ on $\Omega$, weak convergence of measures gives
\[ \int_{\{v <u -\delta \}} (\pi^* \theta +dd^c u)^{m+1}\wedge\eta \leq \int_{\{v <u -\delta \}} (\pi^* \theta +dd^c v)^{m+1}\wedge\eta.\]
The result then follows by taking $\delta\rightarrow 0$ and applying monotone convergence.
\end{proof}

A consequence is that solutions $u\in\PSH(N,\pi^* \theta)\cap C^0(N)$ to the transversal homogeneous Monge-Amp\`{e}re
equation
\begin{equation}\label{eq:HCMA}
(\pi^* \theta +dd^c u)^{m+1}\wedge\eta =0
\end{equation}
are maximal.  By \emph{maximal} we mean that given any neighborhood $U\subset N$ invariant under the Reeb flow, if
$v\in\PSH(U,\pi^* \theta)\cap C^0(\ol{U})$ satisfies $v\leq u$ on $\del U$, then $v\leq u$ on $U$.
For suppose $v$ satisfies $v\leq u$ on $\del U$, then
\[ \tilde{u}=\begin{cases}
u & \text{ on } N\setminus U \\
\max\{u,v\} & \text{ on } U
\end{cases}
\in \PSH(N,\pi^* \theta)\cap C^0(N).\]
We may replace $\pi^* \theta$ with
$\tilde{\theta}=\pi^* \theta +dd^c f =\pi^* \theta +\frac{\sqrt{-1}d\tau\wedge d\ol{\tau}}{|\tau|^2}$ and subtract
$f$ from $u$ and $v$.  Thus we may assume that $u\in\PSH(N,\tilde{\theta})\cap C^0(N)$ with $\tilde{\theta}$ a basic strictly
positive, closed, $(1,1)$-form.  For $\delta>0$ small
\[\begin{split}
\int_{\{u<(1-\delta)\tilde{u}\}} (\delta\tilde{\theta})^{m+1}\wedge\eta & \leq \int_{\{u<(1-\delta)\tilde{u}\}} (\tilde{\theta}+(1-\delta)dd^c \tilde{u})^{m+1}\wedge\eta \\
   & \leq \int_{\{u<(1-\delta)\tilde{u}\}} (\tilde{\theta} +dd^c u)^{m+1}\wedge\eta =0.
\end{split}\]
This clearly implies uniqueness of continuous solutions $u\in\PSH(N,\pi^* \theta)\cap C^0(N)$ to (\ref{eq:HCMA})
with fixed $u|_{\del N} \in C^0(\del N)$.  This also follows from the \emph{strong maximal principle} which we prove
next.

\begin{thm}\label{thm:stron-max}
Let $u,v \in\PSH(N,\pi^* \theta)\cap C^0(N)$.  Suppose that
\[ (\pi^* \theta +dd^c v)^{m+1}\wedge\eta \leq (\pi^* \theta +dd^c u)^{m+1}\wedge\eta \]
and $u\leq v$ on $\del N$.  Then $u\leq v$ on $N$.
\end{thm}
\begin{proof}
Let $u^\epsilon :=\max{u,v+\epsilon}$, then $u^\epsilon =v+\epsilon$ near $\del N$.  The following formula, due
to J-P Demailly: for $s,w \in\PSH\cap L^\infty_{loc}$
\[ (dd^c \max\{s,w \})^{m+1} \geq\mathbb{1}_{\{s\geq w\}}(dd^c s)^{m+1} +\mathbb{1}_{\{s <w\}}(dd^c w)^{m+1}\]
implies that
\[ (\pi^* \theta +dd^c u^\epsilon)^{m+1}\wedge\eta\geq (\pi^* \theta +dd^c v)^{m+1}\wedge\eta.\]

In order to simplify notation, in the following we will denote $\pi^* \theta$ by $\theta$ and $\theta +dd^c u$ by $\theta_u$, etc.

Setting $\phi$ to be $u^\epsilon$ and $\psi$ to be $v+\epsilon$, we have $\phi=\psi$ near $\del N$ and $\phi\geq\psi$ on $N$.
We will show that $\phi=\psi$, which implies the theorem by taking $\epsilon\rightarrow 0$.

Set $\rho=\phi-\psi$, so
\[0\leq\theta_\phi^{m+1} -\theta_\psi^{m+1} =dd^c \rho\wedge\sum_{j=0}^m \theta_{\phi}^j \wedge\theta_{\psi}^{m-j}.\]
By Proposition~\ref{prop:int-parts2}
\[\begin{split}
0 & \leq\int_M dd^c \rho\wedge\sum_{j=0}^m \theta_{\phi}^j \wedge\theta_{\psi}^{m-j}\wedge\eta \\
  & = -\int_M d\rho\wedge d^c \rho \wedge\sum_{j=0}^m \theta_{\phi}^j \wedge\theta_{\psi}^{m-j}\wedge\eta,
\end{split} \]
which implies that
\begin{equation}\label{eq:van-forms}
 d\rho\wedge d^c \rho \wedge \theta_{\phi}^j \wedge\theta_{\psi}^{m-j}\wedge\eta =0,
\end{equation}
for $j=0,\ldots,m$.  We will prove that
\[ d\rho\wedge d^c \rho \wedge\theta^m \eta =0. \]

We will prove inductively in $k=0,\ldots,m$ that
\begin{equation}\label{eq:ind-form}
 d\rho\wedge d^c \rho \wedge\theta_{\phi}^j \wedge\theta_{\psi}^{j}\wedge\eta =0, \text{  for  }i+j=m-k.
\end{equation}
This holds for $k=0$.  Assume that it holds for $0,\ldots,k-1$.
\begin{gather*}
\theta_\phi^i \wedge\theta_\psi^j \wedge\theta^k =\theta_\phi^{i+k} \wedge\theta_\psi^j -dd^c \phi\wedge\alpha, \\
\alpha =\theta_\phi^i \wedge\theta_\psi^j \wedge\sum_{\ell =0}^{k-1} \theta_\phi^\ell \wedge\theta^{k-1-\ell}.
\end{gather*}
Thus we have
\[\begin{split}
d\rho\wedge d^c \rho\wedge\theta_\phi^i \wedge\theta_\psi^j \wedge\theta^k & \leq d\rho\wedge d^c \rho\wedge\bigl(T-dd^c \phi\wedge\alpha \bigr) \\
 & = d\bigl(\rho d^c \rho\wedge T -d^c \phi\wedge\alpha d\rho\wedge d^c \rho\bigr) \\
 & -\rho dd^c \rho\wedge T -d\rho\wedge d^c \phi\alpha\wedge dd^c \rho,
\end{split}\]
where $T=\sum_{j=0}^m \theta_\phi^j \wedge\theta_\psi^{m-j}$.  From (\ref{eq:van-forms}) and Proposition~\ref{prop:int-parts2}
\begin{equation}\label{eq:bound}
\int_N d\rho\wedge d^c \rho\wedge\theta_\phi^i \wedge\theta_\psi^j \wedge\theta^k \wedge\eta
\leq -\int_N d\rho\wedge d^c \phi\wedge\alpha\wedge dd^c \rho\eta.
\end{equation}
But we can bound the right-hand-side by
\begin{multline}\label{eq:split-bd}
 -\int_N d\rho\wedge d^c \phi\wedge\alpha\wedge dd^c \rho\eta \leq \\
 \bigl| \int_N d\rho\wedge d^c \phi\wedge\alpha\wedge\theta_\phi \wedge\eta\bigr| +\bigl|\int_N d\rho\wedge d^c \phi\wedge\alpha\wedge\theta_\psi \bigr|.
\end{multline}
By the Schwartz inequality
\begin{multline*}
\bigl| \int_N  d\rho\wedge d^c \phi\wedge\alpha\wedge\theta_\phi \wedge\eta \bigr| \\
\leq \Bigl( \int_N  d\rho\wedge d^c \rho\wedge\alpha\wedge\theta_\phi \wedge\eta\Bigr)^{\frac{1}{2}}
\Bigl( \int_N  d\phi\wedge d^c \phi\wedge\alpha\wedge\theta_\phi \wedge\eta\Bigr)^{\frac{1}{2}}
\end{multline*}
Since $d\phi\wedge d^c \phi\wedge\alpha\wedge\theta_\phi=0$ by induction,
\[ \int_N  d\rho\wedge d^c \phi\wedge\alpha\wedge\theta_\phi \wedge\eta =0. \]
Similarly,
\[ \int_N  d\rho\wedge d^c \phi\wedge\alpha\wedge\theta_\psi \wedge\eta =0. \]
Thus from (\ref{eq:bound}) and (\ref{eq:split-bd}) we have
\[d\rho\wedge d^c \rho\wedge\theta_\phi^i \wedge\theta_\psi^j \wedge\theta^k \wedge\eta =0, \]
since it is positive.
\end{proof}

\section{Energy functionals}\label{sec:en-funct}

We will define important functionals on the space of potentials $\cH$ and consider their extensions to potentials
of weak regularity.  Fix a Sasakian structure $(\eta,\xi,\Phi,g)$ on $M$.

\subsection{Monge-Amp\`{e}re energy}

We define the \emph{Monge Amp\`{e}re energy}
\begin{equation}\label{eq:M-A-en}
\cE(u) =\frac{1}{m+1}\sum_{j=0}^m \int_M u(\omega_u^T)^j \wedge(\omega^T)^{m-j} \wedge\eta.
\end{equation}
We will show that this is the same as the Monge Amp\`{e}re energy defined in (\ref{eq:M-A-pre}) for $u\in\cH$.
But the definition in (\ref{eq:M-A-en}) extends to $u\in\PSH(M,\omega^T)\cap L^\infty(M)$.  Let $\alpha$ be a basic, closed, $(1,1)$-form on $M$.  We define the $\alpha$-energy to be
\begin{equation}\label{eq:M-A-alph}
\cE^\alpha (u) =\sum_{j=0}^{m-1} \int_M u(\omega_u^T)^j \wedge(\omega^T)^{m-j-1} \wedge\alpha\wedge\eta,
\end{equation}
which is also defined for any $u\in\PSH(M,\omega^T)\cap L^\infty(M)$.
\begin{prop}\label{prop:en-first-dir}
Given $u_1, u_2 \in\PSH(M,\omega^T)\cap C^0(M)$, then
\begin{gather}
\frac{d}{dt} \cE\bigl((1-t)u_1 +tu_2 \bigr)|_{t=0^+} =\int_M (u_2 -u_1)(\omega^T +dd^c u_1)^m \wedge\eta, \\
\frac{d}{dt} \cE^{\alpha}\bigl((1-t)u_1 +tu_2 \bigr)|_{t=0^+} =m\int_M (u_2 -u_1)(\omega^T +dd^c u_1)^{m-1}\wedge\alpha\wedge\eta.
\end{gather}
\end{prop}
\begin{proof}
Let $w=u_2 -u_1$.  Then
\[\begin{split}
\cE\bigl((1-t)u_1 +tu_2 \bigr) & = \sum_{j=0}^m \int_M (u_1 +tw)((1-t)\omega^T_{u_1} +t\omega^T_{u_2})^j \wedge(\omega^T)^{m-j}\wedge\eta \\
   & = \sum_{j=0}^m \int_M u_1 (\omega^T_{u_1})^j \wedge(\omega^T)^{m-j}\wedge\eta \\
   & + t\sum_{j=0}^m \int_M w (\omega^T_{u_1})^j \wedge(\omega^T)^{m-j}\wedge\eta \\
   & + t\sum_{j=1}^m \int_M u_1 j(\omega^T_{u_1})^{j-1}\wedge dd^c w \wedge(\omega^T)^{m-j}\wedge\eta + O(t^2)
\end{split} \]
Then by Proposition~\ref{prop:int-parts}
\begin{multline}
\sum_{j=1}^m \int_M u_1 j(\omega^T_{u_1})^{j-1}\wedge dd^c w \wedge(\omega^T)^{m-j}\wedge\eta \\
 =\sum_{j=1}^m \int_M w dd^c u_1 j(\omega^T_{u_1})^{j-1} \wedge(\omega^T)^{m-j}\wedge\eta \\
 =\sum_{j=1}^m \int_M w j(\omega^T_{u_1})^{j} \wedge(\omega^T)^{m-j}\wedge\eta
 -\sum_{j=0}^{m-1} \int_M w (j+1)(\omega^T_{u_1})^{j} \wedge(\omega^T)^{m-j}\wedge\eta.
\end{multline}
But
\[ \sum_{j=0}^m (\omega^T_{u_1})^j \wedge(\omega^T)^{m-j} +\sum_{j=1}^m j(\omega^T_{u_1})^{j} \wedge(\omega^T)^{m-j}
-\sum_{j=0}^{m-1} (j+1)(\omega^T_{u_1})^{j} \wedge(\omega^T)^{m-j} \]
\[ =(m+1) (\omega^T_{u_1})^{m}.\]
Therefore
\[ \cE\bigl((1-t)u_1 +tu_2 \bigr) = \cE(u_1) +t\int_M w (\omega^T_{u_1})^{m}\wedge\eta +O(t^2), \]
and the first equation follows.  The proof of the second equation is completely analogous.
\end{proof}

\begin{cor}\label{cor:en-cocy}
For $u,v\in\PSH(M,\omega^T)\cap C^0(M)$, we have
\begin{gather}
\cE(u)-\cE(v) =\frac{1}{m+1}\sum_{j=0}^m \int_M (u-v)(\omega^T_u)^j \wedge(\omega^T_v)^{m-j}\wedge\eta, \\
\cE^\alpha (u) -\cE^\alpha (v) =\sum_{j=0}^{m-1} \int_M (u-v)(\omega^T_u)^j \wedge(\omega^T_v)^{m-j-1}\wedge\alpha\wedge\eta.
\end{gather}
\end{cor}
Note that this implies that both functionals are continuous along $C^0$ paths in $\PSH(M,\omega^T)\cap C^0(M)$.
\begin{proof}
Consider the first equation.  Fix $v$ and denote the right-hand expression by $\cF(u)$.  Applying Proposition
~\ref{prop:en-first-dir} with $\omega_v^T$ in place of $\omega^T$ gives
\[\frac{d}{dt}\cF((1-t)v +tu) =\int_M (u-v)(\omega^T_{(1-t)v +tu})^m \wedge\eta =\frac{d}{dt}\cE((1-t)v +tu). \]
And since both $\cF(\cdot)$ and $\cE(\cdot) -\cE(v)$, both vanish at $v$ the result follows.
The second equation follows from a completely analogous argument.
\end{proof}

\begin{cor}\label{cor:en-first-var}
Suppose $\{u_t \}$ is a continuous path in $\PSH(M,\omega^T)\cap C^0(M)$, meaning that $u\in C^0([a,b]\times M)$ and
$u_t \in\PSH(M,\omega^T)\cap C^0(M)$ for $t\in [a,b]$, and which also has $\dot{u}\in C^0([a,b]\times M)$.  Then
\begin{gather}
\frac{d}{dt}\cE(u_t) =\int_M \dot{u}_t (\omega^T +dd^c u_t)^m \wedge\eta, \\
\frac{d}{dt}\cE^{\alpha}(u_t) =m\int_M \dot{u}_t (\omega^T +dd^c u_t)^{m-1} \wedge\alpha\wedge\eta.
\end{gather}
\end{cor}
\begin{proof}
From Corollary~\ref{cor:en-cocy} we have for fixed $t$
\[ \frac{1}{h}\bigl(\cE(u_{t+h})-\cE(u_t)\bigr) =\frac{1}{m+1}\sum_{j=0}^m \int_M \frac{(u_{t+h}-u_t)}{h}(\omega^T_{u_{t+h}})^j \wedge(\omega^T_{u_t})^{m-j}\wedge\eta \]
By assumption $\frac{(u_{t+h}-u_t)}{h}$ converges uniformly to $\dot{u}_t$, thus the formula follows from
the weak converges of measures given in Theorem~\ref{thm:weak-con}.  The same argument gives the second formula.
\end{proof}

We consider the second variation.
\begin{prop}\label{prop:en-sec-var}
Let $U\in\PSH(N,\pi^* \omega^T)\cap C^0(N)$, that is, a subgeodesic if $U$ is $S^1$ invariant.  Then
\begin{gather}
d_\tau d^c_\tau \cE(u_\tau )=\frac{1}{m+1}\int_M (\pi^* \omega^T + dd^c U)^{m+1} \wedge\eta, \\
d_\tau d^c_\tau \cE^{\alpha}(u_\tau )=\int_M (\pi^* \omega^T + dd^c U)^{m}\wedge\alpha\wedge\eta,
\end{gather}
where the integration on the right-hand-side denotes the push-forward of currents under the projection
$M\times A\rightarrow A$, which is the fiberwise integral if the integrand is sufficiently regular.
\end{prop}
\begin{proof}
If $\sigma: N=M\times A\rightarrow A$ is the projection, as a current on $A$ we have
\[\cE(u_\tau )=\sigma_* \bigl( U\sum_{j=0}^m (\pi^* \omega^T +dd^c U)^j \wedge(\pi^* \omega^T)^{m-j}\wedge\eta\bigr). \]
Then
\[\begin{split}
d_\tau d^c_\tau \cE(u_\tau ) & =\sigma_* \bigl( dd^c U\wedge\sum_{j=0}^m (\pi^* \omega^T +dd^c U)^j \wedge(\pi^* \omega^T)^{m-j}\wedge\eta\bigr) \\
 & =\sigma_* \bigl( (\pi^* \omega^T +dd^c U)^{m+1} -(\pi^* \omega^T)^{m+1}  \bigr) \\
 & =\sigma_* \bigl( (\pi^* \omega^T +dd^c U)^{m+1} \bigr).
\end{split}\]
And a similar argument proves the second formula.
\end{proof}

Note that if $U\in\PSH(N,\pi^* \omega^T)\cap C_w^{1,1}(N)$, then the integrands in the proposition are differential
forms with $L^\infty$ coefficients.  Thus the push-forward is ordinary integration along the fibers.

\subsection{Mabuchi K-energy}

The Mabuchi K-energy is a functional is indispensable in K\"{a}hler geometry, as its critical points are constant
scalar curvature metrics.  It has been defined on Sasakian manifolds also~\cite{FOW09}.  The
\emph{Mabuchi K-energy} is the functional $\cM :\cH\rightarrow\R$ with derivative
\begin{equation}\label{eq:K-en-sm}
d\cM|_{\phi}(\psi) =-\int_M \psi\bigl(m\Ric_{\omega^T_{\phi}}\wedge(\omega^T_{\phi})^{m-1} -\ol{S}^T (\omega^T_{\phi})^m \bigr)\wedge\eta,
\end{equation}
where the average
\[\begin{split}
\ol{S}^T & = \frac{\int_M  S^T_g\, d\mu_g}{\int_M d\mu_g}  \\
	   & = \frac{2m\pi c_1(\mathscr{F}_\xi)\cup[\omega^T]^{m-1}}{[\omega^T]^m}
\end{split} \]
is independent of the transversal metric.

Consider the transverse canonical line bundle $\Lambda^{m,0}_b$, whose fiber in any foliation chart $U_\alpha$
as in (\ref{eq:fol-prod}) is the line spanned by $dz^1_\alpha \wedge\ldots\wedge dz^m_\alpha$.  The transversal
K\"{a}hler metric $\omega^T$ induces a metric on $\Lambda^{m,0}_b$.  The metric denoted in each chart by
$e^{-\Psi_\alpha}$, i.e. its value on the canonical section.  Thus
\[\Psi_\alpha =\log\bigl(\frac{(\omega^T)^m }{idz^1_\alpha \wedge d\ol{z}^1_\alpha \wedge\ldots\wedge idz^m_\alpha \wedge d\ol{z}^m_\alpha}\bigr).\]
Then $-\Psi_\alpha$ is the metric on the transversal anti-canonical bundle, or rather its weight, while the metric
in local coordinates is $e^{\Psi_\alpha}$.  More generally, given any measure $\mu$ absolutely continuous
w.r.t. $d\mu_\eta =(\omega^T)^m \wedge\eta$, and invariant under the Reeb flow, we have
$\Psi_\alpha^{\mu} =\log(f) +\Psi_\alpha$, where $f=\frac{d\mu}{d\mu_\eta}$ is the Radon-Nikodym derivative, and
$e^{\Psi_\alpha^{\mu}}$ defines a singular metric on the transverse anti-canonical bundle.

We recall the entropy of a measure.  If $\mu$ is absolutely continuous with respect to $\mu_0$ then
the \emph{entropy} of $\mu$ relative to $\mu_0$ is
\begin{equation}
H_{\mu_0}(\mu) :=\int_M \log\Bigl(\frac{d\mu}{d\mu_0} \Bigr) d\mu.
\end{equation}
In the sequel we will assume $\mu_0$ the measure induced by $d\mu_\eta =(\omega^T)^m \wedge\eta$.

We give an alternative formula for the Mabuchi energy, first given in~\cite{Che00b} in the K\"{a}hler
case.
\begin{equation}\label{eq:K-en}
\cM(u) =\bigl(\ol{S}^T \cE(u) -\cE^{\Ric_{\omega^T}}(u)\bigr) +H_{\mu_0}(d\mu_u),
\end{equation}
where $d\mu_u =(\omega_u^T)^m \wedge\eta$.  The functional $\cM$ is easily seen to be
extended to $\cH_{1,1}=\PSH(M,\omega^T)\cap C_w^{1,1}(M)$.  One can check that this
definition coincides with the above on $\cH$, by differentiating (\ref{eq:K-en}) along smooth
paths, which give the formula in (\ref{eq:K-en-sm}).

\subsection{Convexity of the Mabuchi K-energy}\label{subsec:con-K-en}

We prove Theorem~\ref{thmint:conv-K-ener} in this section, that $\cM(u_t)$ is convex along a weak $C^{1,1}_w$ geodesic.
First we prove that as a current on $A$, $\cM(u_\tau)$ is positive.  The proof follows the same
basic approach as that of R. Berman and B. Berndtsson~\cite{BeBer14} in the K\"{a}hler case, but because one has to be
careful when considering currents and weak differentiation in the transversal holomorphic foliation situation,
we give a complete proof.  The proof is based on properties of local Bergman kernels.

We recall that given a complex manifold $X$,  not assumed to be compact, with a line bundle $L$ and
a bounded metric $\phi$, the Bergman kernel $K_{k\phi}$ is the restriction to the diagonal of the
section of $\bigl( kL +K_X \bigr)\otimes\ol{\bigl( kL +K_X\bigr)}$ giving the reproducing kernel.
For the $L^2$ norm we take
\[ \|\sigma \|_{L^2} =\int_X \sigma\wedge\ol{\sigma}e^{-k\phi}, \]
where $\sigma\in\Gamma( kL +K_X )$.  Contracting by the metric on $kL$ gives a measure
\begin{equation}
\beta_k :=\frac{m!}{k^m} K_{k\phi}e^{-k\phi},
\end{equation}
where $\dim X=m$.

The following convergence result is used.  See~\cite{BeBer14} for the proof.
\begin{thm}[\cite{BeBer14}]\label{thm:Ber-conv}
Let $\B\subset\C^m$ be the unit ball, and let $\phi$ be a plurisubharmonic function defined on a neighborhood of $\B$
and in $C^{1,1}_w$.  If $d\mu$ denotes Lebesgue measure, then on any relatively compact subdomain $E\subset\B$
we have $\beta_k \leq C_E d\mu$.  And after passing to a subsequence
\[ \underset{k\rightarrow\infty}{\lim} \beta_k (x) =( dd^c \phi)^m (x), \]
for almost all $x\in\B$.  Thus $\beta_k \rightarrow( dd^c \phi)^m$ in total variation norm.
\end{thm}

In the following theorem $D\subset\C$ is only required to be a smoothly bounded domain.
\begin{thm}
Let $\{u_\tau\}$ be a $C^{1,1}_w$ solution to the transversal Monge-Amp\`{e}re equation on $N=M\times D$, that is
if $U=\{u_\tau\}$, then $U\in\PSH(N,\pi^* \omega^T)\cap C_w^{1,1}(N)$ and $(\pi^*  \omega^T +dd^c U)^{m+1} \wedge\eta=0$.
Then the Mabuchi functional $\cM(u_\tau)$ is weakly subharmonic with respect to $\tau\in D$.
In particular, if $\{u_t \}$ is a weak geodesic connecting two elements of $\cH$ then $\cM(u_t)$ is weakly convex.
\end{thm}
\begin{proof}
Let $\Psi =\Psi(\tau,x)=\psi_\tau(x)$ be a possibly singular bounded metric on $\Lambda_b^{m,0} =K_{\sF}$.  We define
\begin{equation}\label{eq:Mab-mod}
f^{\Psi}(\tau):=\bigl(\ol{S}^T \cE(u) -\cE^{\Ric_{\omega^T}}(u)\bigr) +\int_M \log\Bigl(\frac{e^{\psi_\tau}}{(\omega^T)^m}\Bigr)(\omega^T_{u_\tau})^m \wedge\eta.
\end{equation}
First we claim that as a current on $D$
\begin{equation}\label{eq:diff-Mab-mod}
dd^c f^{\Psi} (\tau) =\int_M T,\quad T=dd^c \bigl( \Psi(\pi^*  \omega^T +dd^c U)^m \bigr)\wedge\eta,
\end{equation}
where $T$ is a transversal $(m+1,m+1)$ current on $N$.  Part of the problem is to show that $T$ is positive and
thus defines a measure on $N$.  If $v\in C_0^\infty(D)$ and $\varpi:N\rightarrow D$ is the projection,
then by (\ref{eq:diff-Mab-mod}) we mean
\[ \langle dd^c f^{\Psi},v\rangle :=\int_N \Psi(\pi^*  \omega^T +dd^c U)^m \wedge dd^c(\varpi^* v)\wedge\eta. \]
Let $\Psi_j$ be a sequence of uniformly bounded smooth metrics $\Psi_j \rightarrow\Psi$ almost everywhere on $M$
for each $\tau\in D$.  This sequence can be constructed by taking convolutions in each foliation chart as in
(\ref{eq:fol-prod-N}) with a subordinate partition of unity.  In order to get basic metrics $\Psi_j$ one must
average by the torus $T\subset\Aut(\eta,\xi,\Phi,g)$ generated by $\xi$.

Making use of Proposition~\ref{prop:en-sec-var} we compute
\[ dd^c f^{\Psi_j} (\tau) =\int_M T,\quad T_j =dd^c \bigl( \Psi_j (\pi^*  \omega^T +dd^c U)^m \bigr)\wedge\eta, \]
and since $\int_M T_j \rightarrow\int_M T$ weakly as currents by Lebesgue dominated convergence, and likewise
$f^{\Psi_j} \rightarrow f^{\Psi}$ pointwise, we have (\ref{eq:diff-Mab-mod}).

Let $\Theta_j = dd^c \Psi_j$.  We extend $T$ to be a transversal current on $N$ by
\[ \langle T,w\rangle :=\underset{j\rightarrow\infty}{\lim}\int_N w\Theta_j \wedge (\pi^*  \omega^T +dd^c U)^m \wedge\eta,
\quad\text{for  } w\in C_0^\infty(N). \]
This is \textit{a priori} not defined for all test functions.  The proof will proceed by showing when $\Psi$ is
the metric induce by $(\omega_{u_\tau}^T)^m$ that in each
holomorphic foliation chart
\[ \underset{j\rightarrow\infty}{\lim} \Theta_j \wedge (\pi^*  \omega^T +dd^c U)^m =dd^c \bigl( \Psi(\pi^*  \omega^T +dd^c U)^m \bigr) \]
is a positive current, and so $T$ defines a positive Radon measure on $N$.

We want to consider the singular metric $\psi_\tau =\log(\pi^*  \omega^T +dd^c u_\tau )^m$, but an additional problem
is that it is not locally bounded.  Fix $A>0$ and define
\[ \Psi_A =\max\{ \log(\pi^*  \omega^T +dd^c u_\tau )^m ,\chi -A\},\]
where $\chi$ is a fixed continuous metric on $\Lambda_b^{m,0}$ that we will define.
We will prove that the transversal $(m+1,m+1)$ current
\begin{equation}
T_A = dd^c \Psi_A \wedge(\pi^*  \omega^T +dd^c U)^m ,
\end{equation}
satisfies $T_A \geq 0$ if $\chi$ is chosen so that
\begin{equation}\label{eq:chi-boun}
dd^c \chi\geq -k_0 (\pi^*  \omega^T +dd^c U),
\end{equation}
for some $k_0 \in\N$.  This implies that $dd^c f^{\Psi_A} \geq 0$ on $D$.  But dominated convergence implies
$f^{\Psi_A}(\tau)\rightarrow\cM(u_\tau)$ pointwise, thus
\[ dd^c f^{\Psi_A}(\tau)\rightarrow dd^c \cM(u_\tau)\quad\text{weakly} \]
and $dd^c \cM(u_\tau)\geq 0$ as a current.

In order to satisfy (\ref{eq:chi-boun}) first let $\chi_0$ be an arbitrary smooth metric on $\Lambda_b^{m,0}$,
and choose $k_0 \in\N$ large enough that $k_0 \omega^T +dd^c \chi_0 \geq 0$.  Set $\chi=\pi^* \chi_0 -k_0 U$, then
\[\begin{split}
dd^c \chi & =\pi^* dd^c \chi_0 -k_0(\pi^* \omega^T +dd^c U) +k_0 \pi^* \omega^T \\
            & \geq  -k_0(\pi^* \omega^T +dd^c U).
\end{split}\]

We next prove that $T_A \geq 0$ in local foliation charts.  This will follow from a local approximation of the
metric $\Psi_A$ by Bergman densities.  Consider a fixed transversal holomorphic chart
\[ \phi_\alpha :U_\alpha \times V\rightarrow W_\alpha \times V,\quad U_\alpha \times V\subset M\times D.\]
Write
\[ \pi^* \omega^T +dd^c U =dd^c \Psi \]
for a plurisubharmonic function on $W_\alpha \times V$.  We will denote $\phi_\tau =\Phi(\dot,\tau)$.

Let $\beta_{k}$ be the Bergman measure for the Hilbert space of holomorphic functions on the
unit ball $\B\subset W_\alpha\subset\C^m$ with weight $k\phi_\tau$.  Consider the transversal current
\[ T_{A,k}:=dd^c \Psi_{A,k}\wedge\bigl(dd^c \Phi\bigr)^m,\quad \Psi_{A,k}:=\max\{\log\beta_{k\phi_\tau}, \chi-A \}.\]
By Theorem~\ref{thm:Ber-conv} and dominated convergence
\begin{equation}\label{eq:T-weak}
\underset{k\rightarrow\infty}{\lim} T_{A,k} =T_A \quad\text{weakly}
\end{equation}
as transversal currents.  By a result of B. Berndtsson~\cite{Ber06} on the plurisubharmonic variation of Bergman kernels,
$dd^c K_{k\phi_\tau} \geq 0$ on $B\times V$.  Thus
\[ dd^c \log\beta_{k\phi_\tau} \geq -kdd^c \Phi.  \]
Since by (\ref{eq:chi-boun}) we have $dd^c \chi\geq -k_0 dd^c \Phi$, for $k\geq k_0$
\[ dd^c \Psi_{A,k} \geq -kdd^c \Phi.\]
Thus
\[ T_{A,k} =dd^c \Psi_{A,k}\wedge\bigl(dd^c \Phi\bigr)^m \geq -k \bigl(dd^c \Phi\bigr)^{m+1} =0. \]
And finally, by Theorem~\ref{thm:Ber-conv}
\[ e^{\Psi_{A,k}} =\max\Bigl\{\frac{m!}{k^m}K_{k\phi_\tau} e^{-k\phi_\tau} ,e^{-(\chi-A)}\Bigr\}
\rightarrow\Bigl\{(dd^c \phi_tau )^m ,e^{-(\chi-A)}  \Bigr\}\]
pointwise almost everywhere on $M$ for all $\tau\in D$ in a dominated fashion.  Thus (\ref{eq:T-weak}) holds,
and the proof is complete.
\end{proof}

It remains to show that when $\{u_\tau \}$ is a weak geodesic, so $D=A=\{\tau\in\C\ |\ 1\leq |\tau|\leq e \}$
and $u_\tau$ depends only on $|\tau|=e^t$, that $\cM(u_t)$ is convex in the pointwise sense and thus continuous.
\begin{thm}
Suppose that $u_t$ is a weak $C^{1,1}_w$ geodesic.  Then $\cM(u_t)$ is continuous and therefore pointwise convex.
\end{thm}
\begin{proof}
Let $\Psi_A$ be defined as above.  We will prove that $f^{\Psi_A}$ is convex and continuous.  Then $f^{\Psi}$
will be convex by taking $A\rightarrow\infty$.

Let $\kappa_{\epsilon} :\R\rightarrow\R$, for $\epsilon>0$, be a sequence of strictly convex functions with
$\kappa'_{\epsilon} \geq 0$ tending to $Id$ as $\epsilon\rightarrow 0$.  Define $f^{\Psi_A}_\epsilon$ just
as $f^{\Psi_A}$ but with
\[ \log\bigl(\frac{e^{\psi_{A,\tau}}}{(\omega^T)^m}\bigr) \]
replaced by
\[ \kappa_{\epsilon} \Bigl(\log\bigl(\frac{e^{\psi_{A,\tau}}}{(\omega^T)^m}\bigr)\Bigr). \]

We will prove that $f^{\Psi_A}_\epsilon$ is convex for all $\epsilon>0$.  Note that by the argument
in the last theorem $f^{\Psi_A}_\epsilon$ is weakly convex.
Let $\{\sigma_\alpha \}$ be a
partition of unity subordinate to the covering of $M$ by foliation charts $\{U_\alpha \}$.  And
consider the local entropy functions
\[\begin{split}
H_{\alpha} & = \int_M \sigma_{\alpha} \kappa_{\epsilon}\Bigl(\log\bigl(\frac{e^{\psi_{A,\tau}}}{(\omega^T)^m}\bigr)\Bigr)(\omega^T_{\phi_\tau})^m \wedge\eta \\
        & \int_{W_\alpha} \hat{\sigma}_{\alpha} \kappa_{\epsilon}\Bigl(\log\bigl(\frac{e^{\psi_{A,\tau}}}{(\omega^T)^m}\bigr)\Bigr)(\omega^T_{\phi_\tau})^m
\end{split} \]
where $\hat{\sigma}_\alpha$ is $\sigma_\alpha$ integrated along the leaves.
Define $H^{k}_\alpha$ by replacing $\Psi_A$ by its approximation by local Bergman kernels
\[ \Psi_{A,k} =\max\{\log\beta_{k\phi_\tau} ,\chi-A \}.\]
We take $d_\tau d_\tau^c$ of $H^{k}_\alpha$,
\[ d_\tau d_\tau^c H^{k}_\alpha =\int_{W_\alpha} dd^c \Bigl(\hat{\sigma}_{\alpha} \kappa_{\epsilon}\Bigl(\log\bigl(\frac{e^{\psi_{A,\tau}}}{(\omega^T)^m}\bigr)\Bigr)\Bigr)(\omega^T_{\phi_\tau})^m ,\]
and use the plurisubharmonic variation of Bergman kernels~\cite{Ber06},
the strict convexity of $\kappa_{\epsilon}$, and the fact that $dd^c$ commutes with the push-forward
$W_\alpha \times A\rightarrow A$ to get
\begin{equation}
d_\tau d_\tau^c H^{k}_\alpha \geq -C_{\epsilon,\alpha}.
\end{equation}
Thus $H^{k}_\alpha +C_{\epsilon,\alpha} t^2$ is convex since $H^{k}_\alpha$ is continuous.  Taking $k\rightarrow\infty$
we get that $H_{\alpha} +C_{\epsilon,\alpha} t^2$ is convex by dominated convergence.  Summing over $\alpha$
gives that
\[ \int_M \kappa_{\epsilon}\Bigl(\log\bigl(\frac{e^{\psi_{A,\tau}}}{(\omega^T)^m}\bigr)\Bigr)(\omega^T_{\phi_\tau})^m \wedge\eta
+C_\epsilon t^2 \]
is convex and thus continuous.  Therefore $f^{\Psi_A}_\epsilon$ is continuous and poinwise convex, and
$f^{\Psi_A}$ is convex by dominated convergence.
\end{proof}

We now prove Corollary~\ref{cor:Mab-bound}.
\begin{lem}\label{lem:K-en-dif}
Let $\{\phi_t \}$ be the weak $C^{1,1}_w$ geodesic connecting $\phi_0, \phi_1 \in\cH$.  Then
\[ \frac{d}{dt}\cM(\phi_t)|_{t=0^+} \geq -\int_M (S^T_{\phi_0} -\ol{S})\frac{d\phi_t}{dt}|_{t=0^+}(\omega^T_{\phi_0})^m \wedge\eta. \]
\end{lem}
\begin{proof}
Let $\mu =(\omega^T)^m \wedge\eta$.  $H_{\mu}$ is convex on measure with volume $\Vol(\mu)$.
Set $\nu_0 =(\omega_{\phi_0}^T)^m \wedge\eta$ and $\nu_1 =(\omega_{\phi_t}^T)^m \wedge\eta$.
Hence if $\nu_s =s\nu_1 +(1-s)\nu_0$, we have
\[ H_{\mu}(\nu_1)-H_{\mu}(\nu_0)\geq\frac{d}{ds}H_{\mu}(\nu_s)|_{s=0} =\int_M \log\bigl( \frac{\nu_0}{\mu} \bigr)(d\nu_1 -d\nu_0).  \]
Thus
\begin{equation}
\begin{split}
\frac{1}{t}\bigl( H_{\mu}((\omega_{\phi_t}^T)^m \wedge\eta) -  &H_{\mu}((\omega_{\phi_0}^T)^m \wedge\eta)\bigr)  \\
& \geq \int_M \log\Bigl(\frac{(\omega_{\phi_0}^T)^m \wedge\eta}{(\omega^T)^m \wedge\eta} \Bigr)\frac{1}{t}\bigl((\omega^T_{\phi_t})^m -(\omega^T_{\phi_0})^m \bigr)\wedge\eta \\
  & = \int_M \log\bigl(\frac{\nu_0}{\mu}\bigr)\frac{1}{t}dd^c (\phi_t -\phi_0)\wedge\sum_{j=0}^{m-1}(\omega^T_{\phi_t})^{j}\wedge(\omega^T_{\phi_0})^{m-j-1} \wedge\eta \\
  & = \int_M dd^c \log\bigl(\frac{\nu_0}{\mu}\bigr)\frac{1}{t}(\phi_t -\phi_0)\wedge\sum_{j=0}^{m-1}(\omega^T_{\phi_t})^{j}\wedge(\omega^T_{\phi_0})^{m-j-1} \wedge\eta.
\end{split}
\end{equation}
We take the limit of $t$ to zero and apply Theorem~\ref{thm:weak-con}
\[ \underset{t\rightarrow 0^+}{\lim}\frac{1}{t}\bigl( H_{\mu}((\omega_{\phi_t}^T)^m \wedge\eta)
\geq m\int_M \frac{d\phi_t}{dt}|_{t=0^+}\Bigl( \Ric^T_{\omega^T}-\Ric^T_{\omega^T_{\phi_0}}\Bigr)\wedge(\omega^T_{\phi_0})^{m-1} \wedge\eta.  \]
Then the \emph{energy} part of (\ref{eq:K-en}) differentiates as in Corollary~\ref{cor:en-first-var} completing the proof.
\end{proof}
We now can prove Corollary~\ref{cor:Mab-bound}.  By the sub-slope property of convex functions
\[ \cM(\phi_1) -\cM(\phi_0) \geq \frac{d}{dt}\cM(\phi_t)|_{t=0^+}. \]
Thus
\[  \begin{split}
\cM(\phi_1) -\cM(\phi_0) & \geq -\int_M (S^T_{\phi_0} -\ol{S})\frac{d\phi_t}{dt}|_{t=0^+}(\omega^T_{\phi_0})^m \wedge\eta, \\
                        & \geq -\Bigl(\Cal(\phi_0)\Bigr)^{\frac{1}{2}}\Bigl(\int_M \bigl(\frac{d\phi_t}{dt}|_{t=0^+} \bigr)^2 d\mu_{\phi_0}\Bigr)^{\frac{1}{2}}
\end{split}\]
by the Cauchy-Schwartz inequality.  But the distance $d(\phi_0,\phi_1) = \Bigl(\int_M \bigl(\frac{d\phi_t}{dt}|_{t=0^+} \bigr)^2 d\mu_{\phi_0}\Bigr)^{\frac{1}{2}}$.

\section{Uniqueness of constant scalar curvature structures and extremal structures}\label{sec:unique}

\subsection{Automorphism groups}

We will need some facts about the Lie algebra $\hol^T(\xi,J)$ and automorphism groups of $C(M)$
when $\cS(\xi,J)$ admits an extremal structure.  The following was proved in~\cite{BoGaSi08}, and~\cite{vC12b} giving
the last statement.

\begin{thm}\label{thm:aut-S-E}
Let $(\eta,\xi,\Phi,g)\in\mathcal{S}(\xi,J)$ be a Sasaki-extremal structure.
Then we have the semidirect sum decomposition
\begin{equation}
\hol^T(\xi,J) =\mathfrak{a}\oplus\hol^T(\xi,J)_0,
\end{equation}
where $\mathfrak{a}$ is the Lie algebra of parallel, with respect to $g^T$, sections of $\nu(\mathscr{F}_\xi)$.  And we also have
\begin{equation}\label{eq:extr-decomp}
\hol^T(\xi,J)_0 =\fg\oplus J\fg\oplus\bigl(\bigoplus_{\lambda>0}\mathfrak{h}^{\lambda} \bigr),
\end{equation}
where $\fg=\aut(g,\eta,\xi,\Phi)/{\R\xi}$ is the image under $\partial_g^{\#}$ of the imaginary valued functions in $\mathscr{H}_g$
and $\mathfrak{h}^\lambda =\{\ol{X}\in\hol^T(\xi,J)_0 : [\partial^{\#}_g s_g ,\ol{X}]=\lambda\ol{X}\}$
and $\fg\oplus J\fg =C_{\hol^T(\xi,J)_0}(\partial^{\#}_g s_g)$, the centralizer of $\partial^{\#}_g s_g$.

Furthermore, the connected component of the identity $G=\Aut(\eta,\xi,\Phi,g)_0 \subset\Fol(M,\mathscr{F}_\xi,J)$ is a maximal
compact connected subgroup.  And any other maximal compact connected subgroup is conjugate to $G$ in $\Fol(M,\mathscr{F}_\xi,J)$.
\end{thm}

We will need the following definition from~\cite{FOW09}.
\begin{defn}
A complex vector field $X$ on a Sasakian manifold $M$ is \emph{Hamiltonian holomorphic} if
\begin{thmlist}
\item  $X$ is basic, i.e. on each chart $U_\alpha$ it projects, and $\pi_\alpha (X)$ is holomorphic on $V_\alpha$.

\item  $u_X :=\sqrt{-1}\eta(X)$ is a holomorphy potential,
\[ \ol{\del}_b u_X =-\sqrt{-1}X\contr\omega^T.\]
\end{thmlist}
We denote the space of Hamiltonian holomorphic vector fields by $\Ham$.
\end{defn}
We list some useful properties.
\begin{prop}\label{prop:Ham-hol}
\begin{thmlist}
\item  The space of $\Ham$ is a Lie algebra, and $u_{[X,Y]} =X u_Y -Y u_X$ for $X,Y\in\Ham$.\label{prop:Ham1}

\item  There is a Lie algebra isomorphism $\hol^T(\xi,J)_0 \cong\Ham/{\C\xi}$.  \label{prop:Ham2}

\item  $\Ham$ is isomorphic to the Lie algebra of holomorphic vector fields $\tilde{X}$ on $C(M)$ for which
$[\xi,\tilde{X}]=[r\del_r,\tilde{X}]=0$.  \label{prop:Ham3}
\end{thmlist}
It follows that $\Ham$ only depends on $\xi$ and the complex structure on $C(M)$.
\end{prop}
\begin{proof}
\noindent (\ref{prop:Ham1})  Suppose $X,Y\in\Ham$.  Write $X=X_D -\sqrt{-1}u_X \xi$ with $X_D \in\Gamma(D^{1,0})$ and
$X_D =\del^{\#} u_X$ as a section of $\nu(\sF)$ and $Y=Y_D -\sqrt{-1}u_Y$.  Then
\[ \begin{split}
[X,Y] & = [X_D,Y_D] -\sqrt{-1}\bigl([u_X \xi,Y_D] +[X_D,u_Y \xi]  \bigr) \\
        & = [X_D,Y_D] -\sqrt{-1}\bigl(X_D u_Y -Y_D u_X  \bigr).
\end{split}\]
It is easy to check that if we set $u_{[X,Y]} = X_D u_Y -Y_D u_X =X u_Y -Y u_X$, then
$\del^{\#} u_{[X,Y]} =[X_D,Y_D]$.  And further more $\eta([X_D,Y_D])=2\omega^T(X_D,Y,D)=0$.

\noindent (\ref{prop:Ham2})  Suppose $V\in\hol^T(\xi,J)_0$ and $V =\del^{\#}u$ with the holomorphy potential $u$
defined up to a constant.  Then taking $V\in\Gamma(D^{1,0})$, $X=V -\sqrt{-1}u\xi$ is representative of $\Ham$.

\noindent (\ref{prop:Ham3})  If $X\in \Ham$, then a direct computation shows that
$\tilde{X}=X+u_X r\del_r$ is holomorphic on $C(M)$.

Conversely, if $\tilde{X}$ is holomorphic and $[\xi,\tilde{X}]=[r\del_r,\tilde{X}]=0$, then it can be written as
\[ \tilde{X} =X_D -\sqrt{-1}u_X \xi +u_X r\del_r, \]
with $X_D\in\Gamma(D^{1,0})$ and $\xi u_X =0$.  A computation using the warped product formulas shows that
if $V\in\Gamma(D^{0,1})$, then $\nabla^{0,1}_{V} \tilde{X} =0$ implies that $\ol{\del}_b u_X(V) =g^T(X,V)$.
Thus $X_D -\sqrt{-1}u_X \xi\in\Ham$.
\end{proof}

For a given Sasakian structure $(\eta,\xi,\Phi,g)$ we denote by $\mathscr{H}$ the space of
\emph{holomorphy potentials}, that is, the space of $h\in C_b^\infty(M)$ with
\[ \del^{\#} h \in\hol^T(\xi,J)_0 .\]
Note that $h\in\mathscr{H}$ uniquely determines an element of $V_h \in\Ham$, so we will also use
$\del^{\#} h$ to denote this element $V_h \in\Ham$.
From Proposition~\ref{prop:Ham-hol} we have a Lie algebra isomorphism $\mathscr{H}\cong\Ham$.
Thus an element $V\in\Ham$ has a unique potential $h^V$, while as an element of $\hol^T(\xi,J)_0$ the potential
of $V$ is unique up to a constant.  The next observation is important.
\begin{prop}
$h\in\mathscr{H}$ corresponds to $V_h \in\Ham$ with $\re V_h$ preserving $\eta$, and thus $(\eta,\xi,\Phi,g)$,
if and only if up to a constant $h\in C^\infty_b(M,\i\R)$.
\end{prop}

We fix some important notation.  Let $H$ be the connected group with Lie algebra $\Ham$, and let $G\subset H$ be
a maximal connected compact subgroup.  We may assume, by taking $G$-averages, that we have a $G$-invariant
structure $(\eta,\xi,\Phi,g)$.  Let $K\subseteq G$ be connected compact, and let $\fk, \fg$ be their Lie algebras.
We also define
\begin{itemize}
\item  $\fz =Z(\fk )$, the center of $\fk$,

\item  $\fz' =C_{\Ham}(\fk )$, the centralizer of $\fk$ in $\Ham$,

\item  $\fp =N_{\Ham}(\fk )$, the normalizer of $\fk$ in $\Ham$.
\end{itemize}

We denote the corresponding space holomorphy potentials to be $\mathscr{H}^{\fk}, \mathscr{H}^{\fz}$, etc.
Note that $\mathscr{H}^{\fz}$ (respectively $\mathscr{H}^{\fz'}$) consist of $K$-invariant potentials in
$\mathscr{H}^{\fk}$ (respectively $\mathscr{H}$).

We have the injection
\begin{equation}\label{eq:norm-inc}
 \fz' /\fz \hookrightarrow \fp /\fk .
\end{equation}
\begin{prop}\label{prop:norm-inc}
The inclusion (\ref{eq:norm-inc}) is surjective, so we have an isomorphism of Lie algebras
\[  \fz' /\fz \cong \fp /\fk . \]
\end{prop}
\begin{proof}
We claim $\fz' +\fk$.  Let $W\in\Ham$ with $W=\del^{\#} h^W$.
\[ W\in\fp \Leftrightarrow [X,W]\in\fk, \forall X\in\fk.\]
So $\re X(h^W )=h^{[\re X, W]} =\frac{1}{2} h^{[X,W]}\in\mathscr{H}^{\fk}$, since $\re X$ is contact.
Thus for $\gamma\in K$, $\gamma^* h^W -h^W \in\mathscr{H}^{\fk}$.  If we average with respect to Haar
measure on $K$, we get a $K$-invariant $\tilde{h}^W$ so that $\hat{h}^W :=h^W -\tilde{h}^W \in\mathscr{H}^{\fk}$.
We have
\[ W=\del^{\#}\tilde{h}^W + \del^{\#}\hat{h}^W \in \fz' +\fk ,\]
because $\del^{\#}\tilde{h}^W \in\fz'$ since it is $K$-invariant.
\end{proof}

Suppose that $(\eta,\xi,\Phi,g)$ is a Sasakian structure with $G=\Aut(g,\eta,\xi,\Phi)_0 \subset\Fol(M,\mathscr{F}_\xi,J)$
a maximal compact subgroup.  This is equivalent to requiring that $G\subset H$ is maximal compact.
Here we are taking $K=G$ in the above notation.
We have an orthogonal decomposition using the volume form $d\mu_g =(\omega^T)^m \wedge\eta$
\[ C_b^\infty (M)^G =\i\mathscr{H}_g^{\fz} \oplus W_g , \]
with projections
\[ \pi^G : C_b^\infty (M)^G \rightarrow\i\mathscr{H}_g^{\fz}\text{  and  }\pi_g^W :C_b^\infty (M)^G \rightarrow W_g .\]
The \emph{extremal vector field} is defined to be
\begin{equation}\label{eq:extr-vect}
V =\del^{\#} \pi^G(S_g)\in\hol^T(\xi,J)_0 .
\end{equation}
One can check that $V$ is independent of the of the $G$-invariant Sasakian structure in $\mathcal{S}(\xi,J)$.
See~\cite{FuMa95}, where the arguments can be applied in the Sasakian case also.

We define the \emph{reduced scalar curvature} to be
\begin{equation}\label{eq:red-sc}
S_g^G =\pi_g^W (S_g).
\end{equation}
Note that a $G$-invariant structure in $\mathcal{S}(\xi,J)$ has vanishing reduced scalar curvature
if and only if it is Sasaki-extremal.

Next suppose $(\eta,\xi,\Phi,g)$ is Sasaki-extremal.  Then $G=\Aut(\eta,\xi,\Phi,g)_0$ is a maximal compact subgroup
of $H$.  Furthermore, if it is cscS, then $H$ is the complexification of $G$.
Let $P:=N_{H}(G)_0$ be the connected component of the identity of the normalizer of $G$ in $H$.  Define
$Z' =C_H(G)_0$, the identity component of the centralizer of $G$ in $H$, and $Z_0 =Z' \cap G$.
\begin{prop}
Let $(\eta,\xi,\Phi,g)$ be cscS (respectively Sasaki-extremal), then its orbit $\cO =H/G$ (respectively its orbit
under $P$, $\cO_P =P/G =Z'/{Z_0}$) is a symmetric space with the Riemannian structure induced by
$\cO\subset\tilde{\cH}$ (respectively $\cO_P \subset\tilde{\cH}^G$).

It follows that the exponential map
\[ \exp: \Ham/\fg \rightarrow\cO \quad\text{(resp. }\exp:\fp /\fg \rightarrow\cO_P ) \]
is onto.
\end{prop}
\begin{proof}
Suppose $Y\in\Ham$ and $\del^{\#} u^Y =Y$ with $u^Y \in\mathscr{H}\cap C_b^\infty(M,\R)$.
We first prove
\begin{lem}
Let $g_t =\exp(t\re Y)\in H$.  Then $\{ g_t^* \omega^T \ |\ t\in\R \}$ is a geodesic in $\cK$.
\end{lem}
We define a path $\{ \phi_t \ |\ t\in\R\}\subset\cK\cong\tilde{H}$ by $g_t^* \omega^T =\omega^T_{\phi_t}$.
\[ \i\del\ol{\del}\dot{\phi}_t =\frac{d}{dt}\omega^T_{\phi_t} =\cL_{\re Y}\omega^T_{\phi_t} =d(Y\contr\omega^T_{\phi_t})= \i\del\ol{\del}u^Y_{\omega^T_{\phi_t}} \]
So if $\tilde{u}^Y_{\omega^T_{\phi_t}}$ is normalized so that
$\int\tilde{u}^Y_{\omega^T_{\phi_t}}\, (\omega^T_{\phi_t})^m \wedge\eta=0$, then $\dot{\phi}_t =\tilde{u}^Y_{\omega^T_{\phi_t}}$.
Thus we have
\[ \i\del\ol{\del}\ddot{\phi}_t =\i\frac{d}{dt}\del\ol{\del}\dot{\phi}_t = \frac{d}{dt}\cL_{\re Y}\omega^T_{\phi_t}
=\cL_{\re Y}\i\del\ol{\del}\dot{\phi}_t = \i\del\ol{\del}(Y\dot{\phi}_t). \]
Therefore
\[ \ddot{\phi}_t -\frac{1}{2}|d\dot{\phi}_t |^2_{\omega^T_{\phi_t}} =C_t \]
for $C_t \in\R$ for all $t\in\R$.  But we have
\[\begin{split}
0 & = \frac{d}{dt}\cE(\phi_t) =\frac{d}{dt}\langle\dot{\phi}_t ,1\rangle _{\omega^T_{\phi_t}} \\
  & = \langle\frac{D}{dt}\dot{\phi}_t ,1\rangle _{\omega^T_{\phi_t}} =C(t)\Vol(M),
\end{split}\]
thus $C_t =0$, and the lemma is proved.

One can check that for $g\in H$
\[ \cK\ni\omega^T \mapsto g^* \omega^T \in\cK \]
is an isometry of $\cK$.
By Theorem~\ref{thm:aut-S-E}, in the cscS (respectively the Sasaki-extremal) case $(H,G)$ (respectively $(Z',Z_0)$) is
a symmetric pair.  Proposition~\ref{prop:norm-inc} shows that in the Sasaki-extremal case the orbit $\cO\cong \fp/{\fg}$ is
isomorphic to $Z'/{Z_0}$.  The inclusions
\[ \cO\subset\tilde{\cH}\quad (\text{respectively } \cO_P \subset\tilde{\cH}^G )\]
induce a homogeneous Riemannian structure on $\cO$ (respectively $\cO_P$).  Then
in both cases, $\cO \cong H/G$ ($\cO_P \cong Z'/{Z_0}$) has the structure of a Riemannian symmetric space
with the induced metric~\cite[Prop. 3.4]{Hel78}.
\end{proof}

\subsection{Modified Mabuchi functional}

We will consider several modification to the Mabuchi functional.  The first is useful because
the Mabuchi energy $\cM$ is not known to be strictly convex on weak geodesics.  The other
cases give Mabuchi functionals characterizing constant $\alpha$-twisted scalar curvature and Sasaki-extremal
structures.

Let $\mu$ be a smooth, strictly positive volume form on $M$, which for simplicity we assume to be
invariant of the Reeb flow.  We define
\begin{equation}\label{eq:vol-twist}
\cF^{\mu}(u) := \int_M u\, d\mu -c_{\mu}\cE(u),
\end{equation}
where $c_{\mu} >0$ is chosen so that $\cF^{\mu}(1)=0$.  Denote $J_\mu (u) =\int_M u\, d\mu$.
\begin{prop}\label{prop:vol-twist-con}
$\cF^{\mu}$ is strictly convex along weak $C^{1,1}_w$ geodesics.  In fact, if $\{u_t \}$ is a weak geodesic
$J_\mu (u_t)$ is strictly convex, in that if $J_\mu (u_t )$ is affine then $\omega^T_{u_t}$ is constant.

More precisely, if $\omega_{u_t}^T \leq C\omega^T$ and $\nu\geq A (\omega^T)^m \wedge\eta$, then
\[ \frac{d}{dt} J_\mu (u_t)|_{t=b} -\frac{d}{dt} J_\mu (u_t)|_{t=a} \geq \hat{C} d(\omega^T_{u_a},\omega^T_{u_b})^2 ,  \]
$b>a$, where $\hat{C}>0$ depends only on $C,\mu,\omega^t,$ and $M$.
\end{prop}
\begin{proof}
First suppose that $\{u_t \}$ is a smooth subgeodesic, thus $\ddot{u}_t \geq\frac{1}{2}|d\dot{u}_t |^2_{\omega^T_{u_t}}$.
Suppose $\omega^T_{u_t} \leq C\omega^T$ and $\nu\geq A (\omega^T)^m \wedge\eta$, then
\[\begin{split}
\frac{d^2}{dt^2} J_\mu (u_t) & =\int_M \ddot{u}_t \, d\mu\geq\int_M |\ol{\del}\dot{u}_t |^2_{\omega^T_{u_t}} \, d\mu \\
                        & \geq C^{-1} \int_M |\ol{\del}\dot{u}_t |^2_{\omega^T} \, d\mu \\
                        & \geq\frac{A}{c}\int_M |\ol{\del}\dot{u}_t |^2_{\omega^T} (\omega^T)^m \wedge\eta \\
                        & \geq\frac{A}{c}\tilde{C}\int_M |\dot{u}_t -c_t |^2 (\omega^T)^m \wedge\eta,
\end{split}\]
where the last step in the Poincar\'{e} inequality and $c_t$ is the average of $\dot{u}_t$ with respect to
$(\omega^T)^m \wedge\eta$.

Furthermore,
\[ \begin{split}
\int_M |\dot{u}_t -c_t |^2 (\omega^T)^m \wedge\eta & \geq C^{-m} \int_M |\dot{u}_t -c_t |^2 (\omega^T_{u_t})^m \wedge\eta \\
                    & \geq C^{-m} \int_M |\dot{u}_t -b_t |^2 (\omega^T_{u_t})^m \wedge\eta,
\end{split}\]
where $b_t$ is the average of $\dot{u}_t$ with respect to $(\omega^T_{u_t})^m \wedge\eta$.

Combining these and integrating gives
\begin{equation}\label{eq:J-conv}
\begin{split}
 \frac{d}{dt} J_\mu (u_t)|_{t=b} -\frac{d}{dt} J_\mu (u_t)|_{t=a} & \geq\hat{C}\int_a^b \int_M |\dot{u}_t -b_t |^2 (\omega^T_{u_t})^m \wedge\eta \\
                        & = \hat{C} d(\omega^T_{u_a},\omega^T_{u_b})^2 .
\end{split}
\end{equation}

Now suppose that $u_t$ is merely a weak $C^{1,1}_w$ geodesics.  Then there are smooth $\epsilon$-geodesics $u^\epsilon_t $
with $u^\epsilon_t \rightarrow u_t$ as $\epsilon\rightarrow 0$ in the weak-$C^{1,1}_w$ topology.  In particular it converges
uniformly in $C_b^1([0,1]\times M)$, so inequality (\ref{eq:J-conv}) is valid by taking the limit.
\end{proof}

For the second twisting, let $\alpha$ be a basic, closed, strictly positive $(1,1)$-form.  Define
\begin{equation}\label{eq:alph-twist}
\cF^{\alpha}(u) :=\cE^{\alpha}(u) -c_{\alpha}\cE(u),
\end{equation}
where $c_{\alpha} >0$ is chosen so that $\cF^{\alpha}(1)=0$.  From Proposition~\ref{prop:en-sec-var}
$\cF^{\alpha}$ is convex along any weak geodesic in $\PSH(N,\pi^* \omega^T)\cap C^0(N)$.
It is further, strictly convex along smooth geodesics.  If $\{u_t \}$ is a smooth path in $\cH$ a
straightforward calculation gives
\begin{multline}\label{eq:2nd-var-alph}
\frac{d^2}{dt^2}\cF^{\alpha}(u_t)=\int_M \bigl( \ddot{u}_t -\frac{1}{2}|d\dot{u}_t |^2_{\omega^t_{u_t}} \bigr)
\Bigl[\tr_{u_t}\alpha -c_\alpha \Bigr](\omega^T_{u_t})^m \wedge\eta \\
+\int_M \bigl( d\dot{u}_t \wedge d^c \dot{u}_t ,\alpha\bigr)_{\omega^T_{u_t}}(\omega^T_{u_t})^m \wedge\eta.
\end{multline}

\begin{prop}\label{prop:vol-twist-prop}
Suppose $(\eta,\xi,\Phi,g)$ is cscS (respectively Sasaki-extremal), then $\cF^{\mu}$ is proper restricted
to the orbit $\cO$ of $H$ (respectively $\cO_P$ of $P$) and has a unique minimum.
\end{prop}
\begin{proof}
Let $\mu$ and $\nu$ be smooth strictly positive volume forms both with the same total mass on $M$.

We fist consider the case with $\mu = d\mu_{\omega^T}=(\omega^T)^m \wedge\eta$.
Let $\phi\in\cH$ be normalized so that $\int_M \phi\, d\mu_{\omega^T} =0$.
We have $\Delta\phi\geq -m$ and by Green's formula
\begin{equation}\label{eq:upp-bd}
\begin{split}
\phi(x) & = \int_M \Delta\phi(y)G(x,y)\, d\mu_{\omega^T}(y) +\int_M \phi\, d\mu_{\omega^T} \\
        & \leq -m\int_M G(x,y)\, d\mu_{\omega^T}(y)  =C.
\end{split}
\end{equation}
Suppose that $\mu_{\omega^T}$ and $\nu$ have the same total mass on $M$, with $\nu =f\mu_{\omega^T}, f>0$.
Define $\phi_+ =\max(\phi,0)$ and $\phi_- =\max(-\phi,0)$, so $\phi=\phi_+ -\phi_-$.
We have
\begin{equation}\label{eq:int-eq}
\int_M \phi_+ \, d\mu_{\omega^T} =\int_M \phi_- \, d\mu_{\omega^T}.
\end{equation}
Then
\[\begin{split}
|\cF^{\mu_{\omega^T}}(\phi)-\cF^{\nu}(\phi)| & =|\int_M \phi(d\mu_{\omega^T}-d\nu)| \\
                                            & =|\int_M \phi(1-f)d\mu_{\omega^T}| \\
                                            & \leq\int_M |\phi f| d\mu_{\omega^T} \\
                                            & =\int_M (\phi_+  +\phi_- )f d\mu_{\omega^T} \leq\hat{C},
\end{split}\]
where $\hat{C}>0$ depends on $C$, an upper bound on $f$, and $\Vol(\mu_{\omega^T})$.

Then if $\mu$ is any smooth strictly positive measure, with the same total mass as $\nu$, it is easy
to see that there is a constant $C$ so that
\begin{equation}\label{eq:vol-twist-com}
|\cF^{\mu}(\phi)-\cF^{\nu}(\phi)| \leq C,\quad\text{for all } \phi\in\cH.
\end{equation}

If $\mu =\mu_{\omega^T}$ then $\omega^T$ is a critical point of $\cF^{\mu}$ on the orbit $\cO$
(respectively $\cO_P$).  Since the exponential map is onto by
Proposition~\ref{prop:vol-twist-con} $\cF^{\mu}$ is proper on $\cO$ (respectively $\cO_P$).
By (\ref{eq:vol-twist-com}) if $\nu$ is any smooth strictly positive measure with the same total mass, then
$\cF^{\nu}$ is proper also.  Since for $c>0$, $\cF^{c\nu} =c\cF^{\nu}$, the proof is complete.
\end{proof}

\subsection{Uniqueness of cscS structures}\label{subsec:unique-cscS}

In this section we will prove that the convexity of the K-energy implies the uniqueness of cscS metrics
up to automorphisms.  More precisely we prove Corollary~\ref{corint:unique-cscS}.  Since the K-energy
$\cM$ is not known to be strictly convex on weak geodesics we will modify the K-energy functional.

Let $\mu$ be a smooth strictly positive measure invariant under the Reeb flow.  We define the modified
Mabuchi functional
\begin{equation}\label{eq:Mab-mod-vol}
\cM^{t\mu}(\phi) :=\cM(\phi) +t\cF^{\mu}(\phi),\quad \phi\in\cH
\end{equation}
for $t\in[0,\epsilon)$.

The main step in the uniqueness is the following.
\begin{prop}\label{prop:csc-def}
Suppose $(\eta,\xi,\Phi,g)$ is cscS.  Then there exists a structure
$g^* (\eta,\xi,\Phi,g)=(\eta_{\phi_0},\xi,\Phi_{\phi_0},g_{\phi_0})$ in the orbit of $H$ and a smooth path, starting at
$\phi_0$, $\phi\in C^\infty_b([0,\epsilon)\times M)$ with $\phi_t =\phi(t,\cdot)\in\cH$ and $\phi_t$ a critical
point of $\cM^{t\mu}$.
\end{prop}
The proof involves a bifurcation technique first used by S. Bando and T. Mabuchi~\cite{BaMa85} in their uniqueness
proof of K\"{a}hler-Einstein metrics.  More recently, it was used again by X. Chen, M. P\u{a}un and Y. Zeng~\cite{CPZ15}
to prove the uniqueness of extremal K\"{a}hler metrics.

\begin{proof}
By Proposition~\ref{prop:vol-twist-prop} there is a unique minimum $\phi_0 \in\cO$ of $\cF^{\mu}$ restricted
to $\cO$.

Define $\cH^{k+4,\alpha} =\{\phi\in C_b^{k+4,\alpha}\ |\ (\omega^T +dd^c \phi)^m \wedge\eta >0 \}$.
We define for $k\geq 0$ a map
\begin{equation}
\begin{gathered}
\cG :\cH^{k+4,\alpha}(M)\times[0,\epsilon)\longrightarrow C_b^{k,\alpha}(M)d\mu\times[0,\epsilon) \\
\cG(\phi,t) =\Bigl( (\ol{S}-S_{\phi})(\omega^T_{\phi})^m \wedge\eta +t(d\mu-(\omega^T_{\phi})^m \wedge\eta), t \Bigr)
\end{gathered}
\end{equation}
with differential
\begin{equation}
d\cG|_{(\phi_0,0)}(u,a) =\Bigl( \Li_{\phi_0}u (\omega^T_{\phi_0})^m \wedge\eta +a(d\mu-(\omega^T_{\phi_0})^m \wedge\eta ),a\Bigr).
\end{equation}
But
\[ d\cG|_{(\phi_0,0)}: C_b^{k+4,\alpha}(M)\times\R\longrightarrow C_b^{k,\alpha}(M)d\mu\times\R \]
is not surjective or injective if $\mathscr{H}\neq\C$.

Let $d\mu_{\phi_0 } =(\omega^T_{\phi_0})^m \wedge\eta$ and define
\begin{gather*}
\mathscr{H}_{\phi_0} :=\{ u\in C^\infty_b (M)\ |\ \Li_{\phi_0}(u)=0,\ \int u\, d\mu_{\phi_0 } =0 \} \\
\mathscr{H}^\perp_{\phi_0 ,k} :=\{ u\in C^{k,\alpha}_b (M)\ |\ \int uv\, d\mu_{\phi_0 }=0,\ \forall v\in\mathscr{H}_{\phi_0 }, \ \int u\, d\mu_{\phi_0 } =0 \}.
\end{gather*}
We have a splitting
\[  C_b^{k,\alpha}(M)d\mu =\R d\mu_{\phi_0}\oplus\mathscr{H}_{\phi_0}d\mu_{\phi_0}\oplus\mathscr{H}^\perp_{\phi_0 ,k}d\mu_{\phi_0}. \]
Since $\cF^{\mu}$ has $\phi_0$ as a critical point on $\cO$,
\[ d\mu-(\omega^T_{\phi_0})^m \wedge\eta \in\mathscr{H}^\perp_{\phi_0 ,k}d\mu_{\phi_0}. \]
Take the first component
\[ G(\phi,t) =(\ol{S}-S_{\phi})(\omega^T_{\phi})^m \wedge\eta +t(d\mu-(\omega^T_{\phi})^m \wedge\eta).  \]
Define
\[ \Pi:\Bigl( \R\oplus\mathscr{H}_{\phi_0}\oplus\mathscr{H}^\perp_{\phi_0 ,k}\Bigr)\times[0,\epsilon)\rightarrow
\Bigl( \R d\mu_{\phi_0}\oplus\mathscr{H}_{\phi_0}d\mu_{\phi_0}\oplus\mathscr{H}^\perp_{\phi_0 ,k}d\mu_{\phi_0}\Bigr)\times[0,\epsilon),\]
\begin{equation}
\Pi(a+u+w,t)=\bigl( ad\mu_{\phi_0}+ud\mu_{\phi_0}+\pi_2 \circ G(\phi_0 +a+u+w,t),t \bigr),
\end{equation}
where $\pi_2$ is projection onto $\mathscr{H}^\perp_{\phi_0 ,k}d\mu_{\phi_0}$.
Since $d\Pi|_{(0,0)}$ is bijective, we apply the implicit function theorem to get $\epsilon>0$ so that
$\|u\|_{C_b^{k,\alpha}} <\epsilon,\ t<\epsilon$ implies that there is
$\Psi(u,t)\in\mathscr{H}^\perp_{\phi_0,k}$ such that
\begin{equation}\label{eq:imp-1}
\pi_2 \circ G(\phi_0 +u +\Psi(u,t),t)=0.
\end{equation}
Differentiating (\ref{eq:imp-1}) with respect to $t$ gives
\begin{equation}\label{eq:imp-diff}
\Li_{\phi_0}\frac{\del\Psi}{\del t}|_{(0,0)}(\omega^T_{\phi_0})^m \wedge\eta +(d\mu-(\omega^T_{\phi_0})^m \wedge\eta )=0,
\end{equation}
while differentiating with respect to $u$ gives
\begin{equation}
\frac{\del\Psi}{\del u}|_{(0,0)}(v) =0,\quad\text{for all }v\in\mathscr{H}_{\phi_0}.
\end{equation}
Define
\begin{equation}
\begin{gathered}
P(u,t):=\pi_1 \circ G(\phi_0 +u+\Psi(u,t),t)  \\
\tilde{P}(u,t):=\frac{P(u,t)}{t},\ t\in(0,\epsilon),\quad \tilde{P}(u,0):=\underset{t\rightarrow 0^+}{\lim}\frac{P(u,t)}{t}=\frac{\del P}{\del t}|_{(u,0)}.
\end{gathered}
\end{equation}

To complete the proof we will need a technical lemma.  For $\phi\in\cH$ we define a bilinear form $B_\phi (\cdot,\cdot)$
on $C^\infty_b(M)$
\[\begin{split}
B_\phi (u,v) & :=\bigl(\del\ol{\del}v,\del\ol{\del}\Delta_{\phi}u\bigr)_{\phi} +\Delta_{\phi}\bigl(\del\ol{\del}v,\del\ol{\del}u\bigr)_\phi
 +\bigl(\del\ol{\del}\Delta_{\phi}v,\del\ol{\del}u\bigr)_{\phi} \\
& +{g^T}^{p\ol{q}}u_{,\ol{\alpha}p}v_{,\beta\ol{q}}\bigl(\Ric_{\phi}^T \bigr)^{\ol{\alpha}\beta} +{g^T}^{\ol{p}q}u_{,\ol{p}\alpha}v_{,q\ol{\beta}}\bigl(\Ric_{\phi}^T \bigr)^{\alpha\ol{\beta}}
\end{split}\]

We refer the reader to~\cite{CPZ15} for the proof of the following.
\begin{lem}\label{lem:chen}
Let $(\eta_\phi,\xi,\Phi_{\phi},g_\phi)$ be Sasaki-extremal, and $v\in\mathscr{H}_{\phi}$, then
\[ \Li_{\phi}\bigl(\del v,\ol{\del}u\bigr)_{\phi} =\bigl(\del v,\ol{\del}\Li_{\phi}u \bigr)_{\phi} +B_{\phi}(v,u), \]
for all $u\in C^\infty_b(M)$.
\end{lem}

We compute
\[\begin{split}
\tilde{P}(u,0) = & \frac{\del}{\del t} P|_{(u,0)} \\
                & =\pi_1 \Bigl(\Li_{\phi_0 +u+\Psi(u,0)}\frac{\del\Psi}{\del t}|_{(u,0)}(\omega^T_{\phi_0 +u+\Psi(u,0)})^m \wedge\eta \\
                & -(S_{\phi_0 +u+\Psi(u,0)}-\ol{S})\Delta_{\phi_0 +u+\Psi(u,0)}\frac{\del\Psi}{\del t}|_{(u,0)}(\omega^T_{\phi_0 +u+\Psi(u,0)})^m \wedge\eta \\
                & +(d\mu- (\omega^T_{\phi_0 +u+\Psi(u,0)})^m \wedge\eta) \Bigr).
\end{split}\]

Let $w=\frac{\del\Psi}{\del t}|_{(0,0)}$.  We compute the differential in the $\mathscr{H}_{\phi_0}$ direction.
For $v\in\mathscr{H}_{\phi_0}$
\[\begin{split}
\frac{\del}{\del u}\tilde{P}|_{(0,0)}(v) & =\pi_1 \Bigl(\frac{\del}{\del u}\Bigl\{ \Li_{\phi_0 +u+\Psi(u,0)} \frac{\del\Psi}{\del t}|_{(u,0)}(\omega^T_{\phi_0 +u+\Psi(u,0)})^m \wedge\eta  \Bigr\}(v) -\Delta_{\phi_0} v(\omega^T_{\phi_0})^m \wedge\eta \Bigr) \\
            & =\pi_1 \Bigl( -B_{\phi_0}(v,w) (\omega^T_{\phi_0})^m \wedge\eta +\Li_{\phi_0}(w) \Delta_{\phi_0}v (\omega^T_{\phi_0})^m \wedge\eta -\Delta_{\phi_0} v(\omega^T_{\phi_0})^m \wedge\eta  \Bigr)
\end{split}\]
By Lemma~\ref{lem:chen}
\begin{equation}\label{eq:imp-2}
\begin{split}
\frac{\del}{\del u}\tilde{P}|_{(0,0)}(v)= & \pi_1 \Bigl(-\Li_{\phi_0} \bigl(\del v,\ol{\del}w\bigr)_{\phi_0}(\omega^T_{\phi_0})^m \wedge\eta +\bigl(\del v,\ol{\del}\Li_{\phi_0}w \bigr) \\
                            & +\Li_{\phi_0}(w)\Delta_{\phi_0}v(\omega^T_{\phi_0})^m \wedge\eta-\Delta_{\phi_0}v (\omega^T_{\phi_0})^m \wedge\eta  \Bigr).
\end{split}
\end{equation}
We define $f\in C_b^\infty(M),\ f>0,$ by
\begin{equation}\label{eq:vol-ratio}
f =\frac{d\mu}{(\omega^T_{\phi_0})^m \wedge\eta}.
\end{equation}
Then (\ref{eq:imp-diff}) gives
\[ \Li_{\phi_0}w =\frac{-d\mu}{(\omega^T_{\phi_0})^m \wedge\eta} +1 =-f+1. \]
Substituting into (\ref{eq:imp-2}) we have
\begin{equation}\label{eq:imp-3}
 \frac{\del}{\del u}\tilde{P}|_{(0,0)}(v)=-\pi_1 \Bigl( \bigl(\del v,\ol{\del}f\bigr)_{\phi_0}(\omega^T_{\phi_0})^m \wedge\eta
+f(\Delta_{\phi_0} v)(\omega^T_{\phi_0})^m \wedge\eta\Bigr).
\end{equation}
And
\[\begin{split}
\Bigl\langle\frac{\del}{\del u}\tilde{P}|_{(0,0)}(v),v\Bigr\rangle_{L^2} & =-\int_M v\bigl(\del v,\ol{\del}f\bigr)_{\phi_0}(\omega^T_{\phi_0})^m \wedge\eta-\int_M vf(\Delta_{\phi_0} v)(\omega^T_{\phi_0})^m \wedge\eta \\
                & =\int_M \bigl(\del v,\ol{\del}v\bigr)_{\phi_0} f(\omega^T_{\phi_0})^m \wedge\eta =\int_M  \bigl(\del v,\ol{\del}v\bigr)_{\phi_0} d\mu >0
\end{split}\]
unless $v=0$.  Thus $\frac{\del}{\del u}\tilde{P}|_{(0,0)} :\cH_{\phi_0}\rightarrow\cH_{\phi_0}d\mu_{\phi_0}$ is an isomorphism.
By the implicit function theorem there exists $u_t \in\cH_{\phi_0},\ t\in[0,\epsilon)$ so that $\phi_0 +u_t +\Psi(u_t,t)$
is the required solution.
\end{proof}

We now prove Corollary~\ref{corint:unique-cscS}.  Let $(\eta_0 ,\xi,\Phi_0,g_0)$ and $(\eta_1,\xi,\Phi_1,g_1)$
be two cscS structures.  By Proposition~\ref{prop:csc-def} there exists a smooth path
$\{\phi^0_s \ |\ s\in[0,\epsilon)\}\in\cH_{\omega^T}$ such that $\omega^T_{\phi^0_0}$ is in the orbit
$\cO^0$ of $(\eta_0 ,\xi,\Phi_0,g_0)$.  Similarly, there exists a smooth path
$\{\phi^1_s \ |\ s\in[0,\epsilon)\}\in\cH_{\omega^T}$ such that $\omega^T_{\phi^1_0}$ is in the orbit $\cO^1$
of $(\eta_1,\xi,\Phi_1,g_1)$.  For each $s\in [0,\epsilon)$ let $\{ u^s_t \ |\ 0\leq t\leq 1 \}$ be the weak
$C^{1,1}_w$ geodesic with $u^s_0 =\phi^0_s$ and $u^s_1 =\phi^1_s$.

By Lemma~\ref{lem:K-en-dif} we have $\frac{d}{dt}\cM^{s\mu}(u^s_t)|_{t=0^+}\geq 0$, and similarly $\frac{d}{dt}\cM^{s\mu}(u^s_t)|_{t=1^-}\leq 0$.
Thus
\[ \frac{d}{dt}\cM^{s\mu}(u^s_t)|_{t=1^-} - \frac{d}{dt}\cM^{s\mu}(u^s_t)|_{t=0^+}\leq 0. \]
But by strict convexity of $\cM^{s\mu}(u^s_t)$, this must be $>0$, unless $\omega^T_{u^s_0} =\omega^T_{u^s_1}$
for all $s\in(0,\epsilon)$.  Thus $\omega^T_{\phi^0_0} =\omega^T_{\phi^1_0}$.

\subsection{Uniqueness of Sasaki-extremal structures}

Sasaki-extremal structures are characterized as critical points of a modified Mabuchi functional $\cM^V$.
Suppose that $(\eta,\xi,\Phi,g)$ is a Sasakian structure with $G=\Aut(g,\eta,\xi,\Phi)_0 \subset\Fol(M,\mathscr{F}_\xi,J)$
a maximal compact subgroup.  We will define this functional and prove uniqueness from convexity using a deformation
technique as in the cscS case.  Let $V$ be the extremal vector field, defined in (\ref{eq:extr-vect}),
and let $h^V_{\phi}$ be its holomorphy potential with respect to $\omega_{\phi}^T,\ \phi\in\cH^G$,
normalized by $\int_M h^V_{\phi}\, d\mu_{\phi} =0$.
\begin{prop}
Suppose $W\in\hol^T(\xi,J)_0$ has normalized holomophy potential $h^W$ with respect to $\omega^T$.  Then
the normalized holomorphy potential of $W$ with respect to $\omega^T_{\phi}$ is
\[ h_{\phi}^W =h^W +W(\phi).\]
\end{prop}
The proof is straight forward.  See for example~\cite{FuMa95}.  Note that for the extremal vector field $V$
$h^V_{\phi}$ is real valued for $\phi\in\cH^G$, since $\im W \in\fg$.

We define $\cE^V$ on $\cH^G$ as the unique functional with
\[ d\cE^V|_\phi (\dot{\phi}) =\int_M \dot{\phi} h^V_\phi (\omega^T_\phi)^m \wedge\eta. \]
This form is well-known to be closed.  Thus the definition
\begin{equation}\label{eq:extr-en}
\cE^V (\phi) =\int_0^1 \int_M  \dot{\phi}_t h^V_{\phi_t} (\omega^T_{\phi_t})^m \wedge\eta,
\end{equation}
where $\{\phi_t \ |\ 0\leq t\leq 1 \}$ is a smooth path in $\cH^G$ with $\phi_0 =0$ and $\phi_1 =\phi$,
is path independent.  There is a closed form formula for $\cE^V$, found by integrating (\ref{eq:extr-en})
along linear paths,
\begin{equation}\label{eq:extr-en2}
  \cE^V (\phi) = \frac{1}{(m+1)(m+2)}\int_M \phi\sum_{k=0}^m \bigl((n-k+1) h^V +(k+1)h^V_{\phi} \bigr)
  (\omega^T)^k \wedge(\omega_\phi )^{m-k} \wedge\eta.
\end{equation}
One can then uniquely extend $\cE^V$ to $\PSH(M,\omega^T)\cap C^1(M)$ by (\ref{eq:extr-en2}).  This
functional is then continuous in $C^1$ by Theorem~\ref{thm:weak-con}.

\begin{prop}\label{prop:extr-en-geo}
Let $\{ u_t \ |\ 0\leq t\leq 1 \}$ be a $C^{1,1}_w$ weak geodesic between $u_0, u_1 \in\cH$.  Then
$\cE^V$ is linear along $\{ u_t \}$, that is $\frac{d^2}{dt^2}\cE(u_t)=0$.
\end{prop}
\begin{proof}
Let $\{ u^\epsilon_t \ |\ 0\leq t\leq 1 \}$ for $\epsilon>0$ be $\epsilon$-geodesics.  Thus
they are smooth and increase monotonically to $\{ u_t \}$ as $\epsilon\searrow 0$.  The paths
$\{ u^\epsilon_t \ |\ 0\leq t\leq 1 \}$ converge weakly in $C^{1,1}_w$ to $\{ u_t \ |\ 0\leq t\leq 1 \}$;
in particular, they converge in $C^1$.

Given any smooth path $\{ w_t \}$ in $\cH$ we have
\begin{equation}
\frac{d}{dt} \int_M \dot{w}_t h^V_{w_t} (\omega_{w_t}^T )^m \wedge\eta =\int_M \bigl(\ddot{w}_t -\frac{1}{2}|d\dot{w}_t |^2_{\omega^T_{w_t}} \bigr)h^V_{w_t} (\omega_{w_t}^T)^m \wedge\eta.
\end{equation}
Thus
\begin{equation}\label{eq:der-extr-en}
\frac{d}{dt} \int_M \dot{u}^{\epsilon}_t h^V_{u^\epsilon_t} (\omega_{u^\epsilon_t}^T )^m \wedge\eta =\epsilon\int_M h^V_{u^\epsilon_t}(\omega^T)^m\wedge\eta.
\end{equation}
By Theorem~\ref{thm:weak-con}
\begin{equation}
\frac{d}{dt}\cE^V(u^\epsilon_t) =\int_M \dot{u}^\epsilon_t h^V_{u^\epsilon_t} (\omega_{u^\epsilon_t}^T )^m \wedge\eta
\longrightarrow \int_M \dot{u}_t h^V_{u_t} (\omega_{u_t}^T )^m \wedge\eta.
\end{equation}
And from (\ref{eq:der-extr-en}) there is a constant $C>0$ so that
\[ -\epsilon C\leq\frac{d}{dt}\cE^V(u^\epsilon_t)|_{t=b}-\frac{d}{dt}\cE^V(u^\epsilon_t)|_{t=a} \leq\epsilon C. \]
Thus $\frac{d}{dt}\cE^V(u_t)$ is constant.
\end{proof}

We define
\begin{equation}
\cM^V (\phi):=\cM(\phi) +\cE^V(\phi),\quad \phi\in\cH^G.
\end{equation}
Since we do not have strict convexity for $\cM^V$ it will again be necessary to modify it.  So we define
\begin{equation}
\cM^{V,t\mu}(\phi):= \cM^V (\phi) +t\cF^{\mu},\quad \phi\in\cH^G,
\end{equation}
for $t\in[0,\epsilon)$, where we now assume that $\mu$ is a smooth strictly positive $G$-invariant volume form.

Uniqueness of Sasaki-extremal structures will follow from the following, whose proof is similar to
Proposition~\ref{prop:csc-def}
\begin{prop}\label{prop:S-E-def}
Suppose $(\eta,\xi,\Phi,g)$ is Sasaki-extremal  Then there exists a structure
$g^* (\eta,\xi,\Phi,g)=(\eta_{\phi_0},\xi,\Phi_{\phi_0},g_{\phi_0})$ in the orbit of $P$ and a smooth path, starting at
$\phi_0$, $\phi\in C^\infty_b([0,\epsilon)\times M)$ with $\phi_t =\phi(t,\cdot)\in\cH$ and $\phi_t$ a critical
point of $\cM^{V,t\mu}$.
\end{prop}
\begin{proof}
As before, by Proposition~\ref{prop:vol-twist-prop} there is a unique minimum $\phi_0 \in\cH^G$ of $\cF^{\mu}$
restricted to the orbit $\cO_P$ of $P$.

If we denote by $h^V$ the normalized holomorphy potential for $V$ with respect to $\omega^T$, then
$h^V =S_g -\ol{S}$, while for $\phi\in\cH^G$ the holomorphy potential with respect to $\omega^T_{\phi}$ is
\[ h^V_{\phi} =S_g -\ol{S} +V(\phi).  \]

We define $\cH^G_{k+4,\alpha} =\{\phi\in C_b^{k+4,\alpha}(M)^G \ |\ (\omega^T +dd^c \phi)^m \wedge\eta >0 \}$.
And we define a map
\begin{equation}
\begin{gathered}
\cG^V :\cH^G_{k+4,\alpha}\times[0,\epsilon)\longrightarrow C_b^{k,\alpha}(M)^G d\mu\times[0,\epsilon) \\
\cG^V (\phi,t) =\Bigl( (\ol{S}-S_{\phi} +h^V_{\phi})(\omega^T_{\phi})^m \wedge\eta +t(d\mu-(\omega^T_{\phi})^m \wedge\eta), t \Bigr)
\end{gathered}
\end{equation}
with differential
\begin{equation}
d\cG^V |_{(\phi_0,0)}(u,a) =\Bigl( \Li_{\phi_0}u (\omega^T_{\phi_0})^m \wedge\eta +a(d\mu-(\omega^T_{\phi_0})^m \wedge\eta ),a\Bigr).
\end{equation}
But as before
But
\[ d\cG^V|_{(\phi_0,0)}: C_b^{k+4,\alpha}(M)^G\times\R\longrightarrow C_b^{k,\alpha}(M)^G d\mu\times\R \]
is in general not surjective or injective.
Let $d\mu_{\phi_0 } =(\omega^T_{\phi_0})^m \wedge\eta$ and define
\begin{gather*}
\mathscr{H}^G_{\phi_0} :=\{ u\in C^\infty_b (M)^G\ |\ \Li_{\phi_0}(u)=0,\ \int u\, d\mu_{\phi_0 } =0 \} \\
\mathscr{H}^{\perp,G}_{\phi_0 ,k} :=\{ u\in C^{k,\alpha}_b (M)\ |\ \int uv\, d\mu_{\phi_0 }=0,\ \forall v\in\mathscr{H}_{\phi_0 }, \ \int u\, d\mu_{\phi_0 } =0 \}.
\end{gather*}
We have a splitting
\[  C_b^{k,\alpha}(M)^G d\mu =\R d\mu_{\phi_0}\oplus\mathscr{H}^G_{\phi_0}d\mu_{\phi_0}\oplus\mathscr{H}^{\perp,G}_{\phi_0 ,k}d\mu_{\phi_0}. \]
Define
\[ G^V(\phi,t) =(\ol{S}-S_{\phi}+h^V_{\phi})(\omega^T_{\phi})^m \wedge\eta +t(d\mu-(\omega^T_{\phi})^m \wedge\eta).  \]

As before the implicit function theorem gives $\Psi(u,t)\in\mathscr{H}^{\perp,G}_{\phi_0 ,k}$ such that
\[ \pi_2 \circ G^V (\phi_0 +u+\Psi(u,t),t)=0. \]
Differentiating with respect to $t$ gives
\begin{equation}\label{eq:S-E-diff}
\Li_{\phi_0}\frac{\del\Psi}{\del t}|_{(0,0)}(\omega^T_{\phi_0})^m \wedge\eta +(d\mu-(\omega^T_{\phi_0})^m \wedge\eta )=0,
\end{equation}
while differentiating with respect to $u$ gives
\begin{equation}
\frac{\del\Psi}{\del u}|_{(0,0)}(v) =0,\quad\text{for all }v\in\mathscr{H}^G_{\phi_0}.
\end{equation}
Define
\begin{equation}
\begin{gathered}
P(u,t):=\pi_1 \circ G^V(\phi_0 +u+\Psi(u,t),t)  \\
\tilde{P}(u,t):=\frac{P(u,t)}{t},\ t\in(0,\epsilon),\quad \tilde{P}(u,0):=\underset{t\rightarrow 0^+}{\lim}\frac{P(u,t)}{t}=\frac{\del P}{\del t}|_{(u,0)}.
\end{gathered}
\end{equation}

Set $w=\frac{\del\Psi}{\del t}|_{(0,0)}$.  Then the same computation as in Proposition~\ref{prop:csc-def}
\begin{equation}\label{eq:imp-S-E}
\begin{split}
\frac{\del}{\del u}\tilde{P}|_{(0,0)}(v)= & \pi_1 \Bigl(\bigl(\del v,\ol{\del}\Li_{\phi_0}w \bigr) \\
                            & +\Li_{\phi_0}(w)\Delta_{\phi_0}v(\omega^T_{\phi_0})^m \wedge\eta-\Delta_{\phi_0}v (\omega^T_{\phi_0})^m \wedge\eta  \Bigr).
\end{split}
\end{equation}
Using that (\ref{eq:S-E-diff}) gives $\Li_{\phi_0} w =-f+1 \in C^\infty_b(M)^G$, where $f$ is defined in (\ref{eq:vol-ratio}),
we get as before
\[ \Bigl\langle\frac{\del}{\del u}\tilde{P}|_{(0,0)}(v),v\Bigr\rangle_{L^2}=\int_M  \bigl(\del v,\ol{\del}v\bigr)_{\phi_0} d\mu >0 \]
unless $v=0$ for $v\in\mathscr{H}^G_{\phi_0}$.

By the implicit function theorem there exists $u_t \in\mathscr{H}^G_{\phi_0}$ with $P(u_t,t)=0$ and
$G^V(\phi_0 +u_t +\Phi(u_t,t),t)=0$.
\end{proof}

We now prove Corollary~\ref{corint:unique-Sasak-ext}.  First we apply the last part of Theorem~\ref{thm:aut-S-E}
act by an element of $\Fol(M,\mathscr{F}_\xi,J)$ so that
$\Aut(\eta_0,\xi,\Phi_0,g_0)_0 = \Aut(\eta_1,\xi,\Phi_1,g_1)_0 =G$.
The rest of the proof of Corollary~\ref{corint:unique-Sasak-ext} is nearly identical to that of
Corollary~\ref{corint:unique-cscS}.  One just makes use that $\cE^V$ is affine along weak geodesics.

We now prove Corollary~\ref{cor:Mab-rel-bound}.  Let $\{\phi_t \}$ be the weak $C_w^{1,1}$ geodesic connecting
$\phi_0,\phi_1 \in\cH^G$.  Lemma~\ref{lem:K-en-dif} gives
\begin{equation}\label{eq:diff-rel-K-en}
\frac{d}{dt}^V \cM(\phi_t)|_{t=0^+} \geq -\int_M (S^T_{\phi_0} -\ol{S}^T-h^V_{\phi_0})\frac{d\phi_t}{dt}|_{t=0^+}(\omega^T_{\phi_0})^m \wedge\eta.
\end{equation}
Note that for $\phi\in\cH^G$, if we denote by $S^G_\phi$ the reduced scalar curvature of $(\eta_\phi ,\xi,\Phi_\phi,g_\phi )$, then
\[ S^G_\phi =S^T_{\phi} -\ol{S}^T-h^V_{\phi}.\]
The rest of the proof follows from (\ref{eq:diff-rel-K-en}) just as in the proof of Corollary~\ref{cor:Mab-bound}.

\begin{rmk}
Corollary~\ref{cor:Mab-rel-bound} can be easily generalized to any compact group $K\subset G$ such that
$\im V \in K$.  The inequality then applies to two potentials $\phi_0,\phi_1 \in\cH^K$ and in the righthand
side the Calabi functional $\Cal^K_{M,\xi}(\phi) =\int_M \bigl( S_\phi^K \bigr)^2 \, d\mu_\phi$,
where the reduced scalar curvature $S_\phi^K $ is defined as in (\ref{eq:red-sc}).
\end{rmk}

We now prove Corollary~\ref{cor:unique-gen}.  Let $(\eta_0,\xi_0,\Phi_0, g_0),(\eta_1,\xi_1,\Phi_1,g_1)$ be two
Sasaki-extremal structures with Reeb foliation $\mathscr{F}$ with its given transversely holomorphic structure.

We consider the \emph{leafwise cohomology} defined by the complex
\[ 0\rightarrow C^\infty(M)\overset{d^\mathscr{F}}{\rightarrow}C^\infty\bigl(\Lambda^1 \mathscr{F}\bigr)\rightarrow 0,\]
where for $f\in C^\infty(M),\ d^\mathscr{F} f=df|_{T\mathscr{F}}$.  Any contact form $\eta$ of a Sasakian structure
compatible with $\mathscr{F}$, with its holomorphic structure defines a class $[\eta]_{\mathscr{F}} \in H^1(\mathscr{F})$.
$H^1(\mathscr{F})$ can be identified with the component $E_1^{0,1}(\mathscr{F})$ of the $E_1$-term of the spectral
sequence associated with $\mathscr{F}$ (cf.~\cite{KaTo83}).  Then $E_2^{0,1}(\mathscr{F})\subseteq E_1^{0,1}(\mathscr{F})$
consists of $d_1$-closed elements.  By~\cite[Cor. 4.7]{KaTo83} $\dim E_2^{0,1}(\mathscr{F})=1$.  Since $d_1 [\eta]_{\mathscr{F}} =0$
for any contact form of a compatible Sasakian structure, $E_2^{0,1}(\mathscr{F})$ is generated by $[\eta]_{\mathscr{F}}$.
Thus $[\eta_0 ]_{\mathscr{F}} =a[\eta_1 ]_{\mathscr{F}}$ for $a\neq 0$, and there exists $f\in C^\infty(M)$ with
\[ \eta_0 -a\eta_1 -df |_{T\mathscr{F}} =0.  \]

Define $\eta_t =\eta_0 +t(a\eta_1 -\eta_0)$ and define $X_t \in C^\infty(T\mathscr{F})$ by $\eta_t(X_t)=f$.
Then if $\Psi_t$ is the flow of $X_t$, an easy computation shows
\[ \frac{d}{dt}\Psi_t^* \eta_t |_{T\mathscr{F}} =0. \]
So $\eta_0 |_{T\mathscr{F}} = a\Psi_1^*\eta_1  |_{T\mathscr{F}}$.
It follows that $a>0$ and $\Psi_1^* (\eta_1,\xi_1,\Phi_1,g_1) =(a^{-1}\hat{\eta},a\xi_0,\hat{\Phi},\hat{g})$.

Given a Sasakian structure $(\eta,\xi,\Phi,g)$ the \emph{transverse homothety} by $a>0$ is the Sasakian structure
$(\eta_a,\xi_a,\Phi,g_a)$ with
\[ \eta_a =a\eta,\ \xi_a =a^{-1}\xi,\ \Psi_a =\Psi,\ g_a =ag+(a^2 -a)\eta\otimes\eta. \]
Since $g^T_a =ag^T$ the transverse Ricci curvature is unchanged $\Ric_{g^T_a} = \Ric_{g^T}$.  So a transverse
homothety of a cscS (respectively Sasaki-extremal) structure is cscS (respectively Sasaki-extremal).

Preforming a transverse homothety by $a>0$ we get $(\hat{\eta},\xi_0,\hat{\Phi},\hat{g}_a)$.
By Lemma~\ref{lem:trans-def}
\[ \hat{\eta} =\eta_0 +2d^c \phi +d\psi +\alpha, \]
where $\alpha\in\mathcal{H}^1_{g^T}$ is a transversal harmonic 1-form.
The exact component $d\psi$ is just given by a gauge transformation.  More precisely, if $b=\exp(-\psi\xi_0)$,
then $b^*\hat{\eta} =\eta_0 +2d^c \phi +\alpha$.  By Corollary~\ref{corint:unique-Sasak-ext} there is a
$g\in\Fol(\mathscr{F},J)$ with $g^* b^* \frac{1}{2}d\hat{\eta}=\frac{1}{2}d\eta_0$.

\subsection{Results on the $\alpha$-twisted case}

We prove uniqueness results for twisted constant scalar curvature metrics and, more generally, twisted
extremal metrics.  These metrics have been of interest in K\"{a}hler geometry~\cite{Sto09b,Che15} as a possible
approach to the general existence problem of constant scalar curvature metrics and their connection to
geometric stability.

In this section $\alpha$ will be any smooth, basic, strictly positive $(1,1)$-form on $M$.
A Sasakian structure $(\eta,\xi,\Phi,g)$ has constant $\alpha$-twisted scalar curvature if
\[ S_g^T -\tr_{\omega^T}\alpha =C_\alpha,\]
where $C_\alpha$ is a constant that depends only on the Sasakian structure and the basic cohomology
class $[\alpha] \in H^2_b(M,\R)$.  These metrics are precisely the Sasakian structures in
$\cS(\xi,J)$ which are critical points of
\begin{equation}\label{eq:K-E-twist}
\cM^{\alpha}(\phi) =\cM(\phi) +\cF^{\alpha}(\phi),\quad \phi\in\cH.
\end{equation}

\begin{thm}\label{thm:unique-acsc}
Any two $\alpha$-twisted constant scalar curvature structures in $\cS(\xi,J)$ have the same transversal
K\"{a}hler metric.
\end{thm}
\begin{proof}
As before, we consider a perturbed Mabuchi functional
\begin{equation}\label{eq:K-E-alpha-per}
\cM^{\alpha, t\mu} := \cM^{\alpha} +t\cF^{\mu}.
\end{equation}
Suppose that $\phi_0 \in\cH$ is such that $\int\phi_0 d\mu_{\phi_0} =0$ and $\omega^T_{\phi_0}$ has constant
$\alpha$-twisted scalar curvature.
Define
\[ G :\tilde{\cH}^{k+4,\alpha} \rightarrow \tilde{C}^{k,\alpha}_b(M)d\mu_{\phi_0}, \]
where $\tilde{\cH}^{k+4,\alpha}=\{\phi\in C_b^{k+4,\alpha}\ |\ (\omega^T +dd^c \phi)^m \wedge\eta >0,\ \int\phi\, d\mu_{\phi_0} =0  \}$, and $\tilde{C}^{k,\alpha}_b(M)d\mu_{\phi_0}$ is the subspace with integral zero, by
\[ G(\phi)=\bigl(\ol{S}^T -S_{\phi}^T \bigr)(\omega^T_{\phi})^m \wedge\eta +m\alpha\wedge(\omega^T_{\phi})^{m-1}\eta -C_{\alpha}(\omega^T_{\phi})^m \wedge\eta.  \]
We compute
\begin{equation}
dG|_{\phi_0} (u) =\Bigl( \Li_{\phi_0}(u) +\i\bigl(\ol{\del}u,\ol{\del}^* \alpha\bigr)_{\phi_0} -\bigl( \i\del\ol{\del}u,\alpha \bigr)_{\phi_0} \Bigr)(\omega^T_{\phi_0})^m \wedge\eta,
\end{equation}
where we have used that $S^T_{\phi_0}$ satisfies the $\alpha$-twisted cscS equation.
Integrating by parts gives
\[\Bigl\langle dG|_{\phi_0} (u),u \Bigr\rangle_{L^2}=\int_M u\Li_{\phi_0}(u)\, d\mu_{\phi_0}
+\int_M \bigl(\i\del u\wedge\ol{\del},\alpha \bigr)_{\phi_0} \, d\mu_{\phi_0} >0, \]
unless $u$ is constant.
Since $dG|_{\phi_0}$ transversely is elliptic, it is Fredholm.  And since it differs from $\Delta^2_{\phi_0}$
by a compact operator its index is zero, and is therefore an isomorphism.

Define
\[ \cG :\tilde{\cH}^{k+4,\alpha}\times[0,\epsilon) \rightarrow \tilde{C}^{k,\alpha}_b(M)d\mu_{\phi_0}\times[0,\epsilon), \]
\[ \cG(\phi,t)=\bigl(G(\phi)+ t(d\mu -(\omega^T_{\phi})^m \wedge\eta),t \bigr).\]
Thus
\[d\cG|_{\phi_0}(u,a) =\bigl( dG|_{\phi_0}(u) +a(d\mu -(\omega^T_{\phi})^m \wedge\eta),a \bigr) \]
is an isomorphism.  The implicit function theorem then gives a path $\{\phi_t \ |\ t\in[0,\epsilon) \}$
in $\cH$ with $\phi_t$ a critical point of $\cM^{\alpha,t\mu}$.
The proof is completed just as that of Corollary~\ref{corint:unique-cscS} using that $\cF^\alpha$ is convex
along weak geodesics.
\end{proof}

A Sasakian structure $(\eta,\xi,\Phi,g)$ is $\alpha$-twisted extremal if
\[ S_g^T -\tr_{\omega^T}\alpha =\mathscr{H}_g .\]
Thus the left hand side is $h^V +C_\alpha$, where $h^V$ is the normalized holomorphy potential.
Since it is a real potential, $\del^{\#}h^V =V\in\Ham$ has $\im V$ preserving $(\eta,\xi,\Phi,g)$.
\begin{lem}
We have $\cL_{\im V} \alpha =0$.  Thus if $K\subset\Aut(\eta,\xi,\Phi,g)$ is the closure of
$\{\exp(s\xi),\exp(t\im V)\ |\ s,t\in\R \}$.  Then $\alpha$ is $K$-invariant.
\end{lem}
\begin{proof}
Averaging $\alpha$ with respect to $K$ gives a $K$-invariant $\hat{\alpha}$ with $[\hat{\alpha}]=[\alpha]$.
So there exists $\psi\in C_b^\infty(M)$ with $\alpha =\hat{\alpha} +dd^c \psi$.  Then the
$\alpha$-twisted extremal equation becomes
\[ S_g -\tr_{\omega^T}\hat{\alpha} -\Delta_{\omega^T} \psi =h^V +C_\alpha. \]
Taking the Lie derivative gives $\Delta_{\omega^T}\im V(\psi) =0$, which implies
$\im V(\psi) =0$.
\end{proof}

We are able to prove a partial uniqueness result for $\alpha$-twisted extremal structures
by modifying the proof of Theorem~\ref{thm:unique-acsc}.
\begin{thm}\label{thm:unique-aextr}
Any two $\alpha$-twisted extremal structures in $\cS(\xi,J)$ with
$\del^{\#}( S_g^T -\tr_{\omega^T}\alpha)=V$ have the same transversal K\"{a}hler metric.
\end{thm}
The proof goes through just as Theorem~\ref{thm:unique-acsc}, mutatis mutandis.
Unlike the untwisted extremal case there is no reason for the vector field $V$ to
be an invariant of the polarization.

We remark that versions of Propositions~\ref{prop:csc-def} (respectively~\ref{prop:S-E-def}) involving
the deformed Mabuchi functional $\cM^{t\alpha}=\cM +t\cF^\alpha$
(respectively $\cM^{V,t\alpha}=\cM^V +\cF^\alpha$) can be proved following the same method as the
above proofs, as long as $[\alpha]=[\omega^T ]$ in basic cohomology.

\bibliographystyle{amsplain}

\begin{thebibliography}{10}

\bibitem{AFHS98}
B.~S. Acharya, J.~M. Figueroa-O'Farrill, C.~M. Hull, and B.~Spence,
  \emph{Branes at conical singularities and holography}, Adv. Theor. Math.
  Phys. \textbf{2} (1998), no.~6, 1249--1286 (1999). \MR{1693624 (2001g:53081)}

\bibitem{BaMa85}
Shigetoshi Bando and Toshiki Mabuchi, \emph{Uniqueness of {E}instein {K}\"ahler
  metrics modulo connected group actions}, Algebraic geometry, {S}endai, 1985,
  Adv. Stud. Pure Math., vol.~10, North-Holland, Amsterdam, 1987, pp.~11--40.
  \MR{946233 (89c:53029)}

\bibitem{BeTa76}
Eric Bedford and B.~A. Taylor, \emph{The {D}irichlet problem for a complex
  {M}onge-{A}mp\`ere equation}, Invent. Math. \textbf{37} (1976), no.~1, 1--44.
  \MR{0445006 (56 \#3351)}

\bibitem{BBGZ13}
Robert~J. Berman, S{\'e}bastien Boucksom, Vincent Guedj, and Ahmed Zeriahi,
  \emph{A variational approach to complex {M}onge-{A}mp\`ere equations}, Publ.
  Math. Inst. Hautes \'Etudes Sci. \textbf{117} (2013), 179--245. \MR{3090260}

\bibitem{Ber06}
Bo~Berndtsson, \emph{Subharmonicity properties of the {B}ergman kernel and some
  other functions associated to pseudoconvex domains}, Ann. Inst. Fourier
  (Grenoble) \textbf{56} (2006), no.~6, 1633--1662. \MR{2282671 (2007j:32033)}

\bibitem{Ber15}
\bysame, \emph{A {B}runn-{M}inkowski type inequality for {F}ano manifolds and
  some uniqueness theorems in {K}\"ahler geometry}, Invent. Math. \textbf{200}
  (2015), no.~1, 149--200. \MR{3323577}

\bibitem{Be87}
Arthur~L. Besse, \emph{Einstein manifolds}, Ergebnisse der Mathematik und ihrer
  Grenzgebiete (3) [Results in Mathematics and Related Areas (3)], vol.~10,
  Springer-Verlag, Berlin, 1987. \MR{867684 (88f:53087)}

\bibitem{BloKo07}
Zbigniew B{\l}ocki and S{\l}awomir Ko{\l}odziej, \emph{On regularization of
  plurisubharmonic functions on manifolds}, Proc. Amer. Math. Soc. \textbf{135}
  (2007), no.~7, 2089--2093 (electronic). \MR{2299485 (2008a:32029)}

\bibitem{Boy11}
Charles~P. Boyer, \emph{Extremal {S}asakian metrics on {$S^3$}-bundles over
  {$S^2$}}, Math. Res. Lett. \textbf{18} (2011), no.~1, 181--189. \MR{2756009
  (2012d:53132)}

\bibitem{BoGa05}
Charles~P. Boyer and Krzysztof Galicki, \emph{Sasakian geometry, hypersurface
  singularities, and {E}instein metrics}, Rend. Circ. Mat. Palermo (2) Suppl.
  (2005), no.~75, 57--87. \MR{2152356 (2006i:53065)}

\bibitem{BoGa08}
\bysame, \emph{Sasakian geometry}, Oxford Mathematical Monographs, Oxford
  University Press, Oxford, 2008. \MR{2382957 (2009c:53058)}

\bibitem{BGK05}
Charles~P. Boyer, Krzysztof Galicki, and J{\'a}nos Koll{\'a}r, \emph{Einstein
  metrics on spheres}, Ann. of Math. (2) \textbf{162} (2005), no.~1, 557--580.
  \MR{2178969 (2006j:53058)}

\bibitem{BoGaSi08}
Charles~P. Boyer, Krzysztof Galicki, and Santiago~R. Simanca, \emph{Canonical
  {S}asakian metrics}, Comm. Math. Phys. \textbf{279} (2008), no.~3, 705--733.
  \MR{2386725 (2009a:53077)}

\bibitem{BoTF13}
Charles~P. Boyer and Christina~W. T{\o}nnesen-Friedman, \emph{Extremal
  {S}asakian geometry on {$T^2\times S^3$} and related manifolds}, Compos.
  Math. \textbf{149} (2013), no.~8, 1431--1456. \MR{3103072}

\bibitem{BoTF14}
\bysame, \emph{Extremal {S}asakian geometry on {$S^3$}-bundles over {R}iemann
  surfaces}, Int. Math. Res. Not. IMRN (2014), no.~20, 5510--5562. \MR{3271180}

\bibitem{BoTF15}
\bysame, \emph{The {S}asaki join and admissible {K}\"ahler constructions}, J.
  Geom. Phys. \textbf{91} (2015), 29--39. \MR{3327046}

\bibitem{ChTi08}
X.~X. Chen and G.~Tian, \emph{Geometry of {K}\"ahler metrics and foliations by
  holomorphic discs}, Publ. Math. Inst. Hautes \'Etudes Sci. (2008), no.~107,
  1--107. \MR{2434691 (2009g:32048)}

\bibitem{Che15}
Xiu~Xiong Chen, \emph{On the existence of constant scalar curvature k\"{a}hler
  metrics: a new perspective}, arXiv:1506.06423v1, 2015.

\bibitem{Che00b}
Xiuxiong Chen, \emph{On the lower bound of the {M}abuchi energy and its
  application}, Internat. Math. Res. Notices (2000), no.~12, 607--623.
  \MR{1772078 (2001f:32042)}

\bibitem{Che00a}
\bysame, \emph{The space of {K}\"ahler metrics}, J. Differential Geom.
  \textbf{56} (2000), no.~2, 189--234. \MR{1863016 (2003b:32031)}

\bibitem{CFO08}
Koji Cho, Akito Futaki, and Hajime Ono, \emph{Uniqueness and examples of
  compact toric {S}asaki-{E}instein metrics}, Comm. Math. Phys. \textbf{277}
  (2008), no.~2, 439--458. \MR{2358291 (2008j:53076)}

\bibitem{CoMi14}
Tobias~Holck Colding and William~P. Minicozzi, II, \emph{On uniqueness of
  tangent cones for {E}instein manifolds}, Invent. Math. \textbf{196} (2014),
  no.~3, 515--588. \MR{3211041}

\bibitem{Dem12}
Jean-Pierre Demailly, \emph{Complex analytic and algebraic geometry}, Book
  avaailable at www-fourier.ujf-grenoble.fr/~demailly/books.html, 2012.

\bibitem{Don97}
S.~K. Donaldson, \emph{Remarks on gauge theory, complex geometry and
  {$4$}-manifold topology}, Fields {M}edallists' lectures, World Sci. Ser. 20th
  Century Math., vol.~5, World Sci. Publ., River Edge, NJ, 1997, pp.~384--403.
  \MR{1622931 (99i:57050)}

\bibitem{Don99}
\bysame, \emph{Symmetric spaces, {K}\"ahler geometry and {H}amiltonian
  dynamics}, Northern {C}alifornia {S}ymplectic {G}eometry {S}eminar, Amer.
  Math. Soc. Transl. Ser. 2, vol. 196, Amer. Math. Soc., Providence, RI, 1999,
  pp.~13--33. \MR{1736211 (2002b:58008)}

\bibitem{DoSu14}
Simon Donaldson and Song Sun, \emph{Gromov-{H}ausdorff limits of {K}\"ahler
  manifolds and algebraic geometry}, Acta Math. \textbf{213} (2014), no.~1,
  63--106. \MR{3261011}

\bibitem{FuMa95}
Akito Futaki and Toshiki Mabuchi, \emph{Bilinear forms and extremal {K}\"ahler
  vector fields associated with {K}\"ahler classes}, Math. Ann. \textbf{301}
  (1995), no.~2, 199--210. \MR{1314584 (95m:32039)}

\bibitem{FOW09}
Akito Futaki, Hajime Ono, and Guofang Wang, \emph{Transverse {K}\"ahler
  geometry of {S}asaki manifolds and toric {S}asaki-{E}instein manifolds}, J.
  Differential Geom. \textbf{83} (2009), no.~3, 585--635. \MR{2581358
  (2011c:53091)}

\bibitem{GuaZha12}
Pengfei Guan and Xi~Zhang, \emph{Regularity of the geodesic equation in the
  space of {S}asakian metrics}, Adv. Math. \textbf{230} (2012), no.~1,
  321--371. \MR{2900546}

\bibitem{Hel78}
Sigurdur Helgason, \emph{Differential geometry, {L}ie groups, and symmetric
  spaces}, Pure and Applied Mathematics, vol.~80, Academic Press, Inc.
  [Harcourt Brace Jovanovich, Publishers], New York-London, 1978. \MR{514561
  (80k:53081)}

\bibitem{JinZha15}
Xishen Jin and Xi~Zhang, \emph{Uniqueness of constant scalar curvature
  {S}asakian metrics}, arXiv:1509.06522v2, 2015.

\bibitem{KaTo83}
Franz~W. Kamber and Philippe Tondeur, \emph{Duality for {R}iemannian
  foliations}, Singularities, {P}art 1 ({A}rcata, {C}alif., 1981), Proc.
  Sympos. Pure Math., vol.~40, Amer. Math. Soc., Providence, RI, 1983,
  pp.~609--618. \MR{713097 (85e:57030)}

\bibitem{Kol05}
J{\'a}nos Koll{\'a}r, \emph{Einstein metrics on five-dimensional {S}eifert
  bundles}, J. Geom. Anal. \textbf{15} (2005), no.~3, 445--476. \MR{2190241
  (2007c:53056)}

\bibitem{Mab87}
Toshiki Mabuchi, \emph{Some symplectic geometry on compact {K}\"ahler
  manifolds. {I}}, Osaka J. Math. \textbf{24} (1987), no.~2, 227--252.
  \MR{909015 (88m:53126)}

\bibitem{Mal98}
Juan Maldacena, \emph{The large {$N$} limit of superconformal field theories
  and supergravity}, Adv. Theor. Math. Phys. \textbf{2} (1998), no.~2,
  231--252. \MR{1633016 (99e:81204a)}

\bibitem{MSY08}
Dario Martelli, James Sparks, and Shing-Tung Yau, \emph{Sasaki-{E}instein
  manifolds and volume minimisation}, Comm. Math. Phys. \textbf{280} (2008),
  no.~3, 611--673. \MR{2399609 (2009d:53054)}

\bibitem{MoPl99}
David~R. Morrison and M.~Ronen Plesser, \emph{Non-spherical horizons. {I}},
  Adv. Theor. Math. Phys. \textbf{3} (1999), no.~1, 1--81. \MR{1704143
  (2000d:83122)}

\bibitem{NiSe12}
Yasufumi Nitta and Ken'ichi Sekiya, \emph{Uniqueness of {S}asaki-{E}instein
  metrics}, Tohoku Math. J. (2) \textbf{64} (2012), no.~3, 453--468.
  \MR{2979292}

\bibitem{ONe66}
Barrett O'Neill, \emph{The fundamental equations of a submersion}, Michigan
  Math. J. \textbf{13} (1966), 459--469. \MR{0200865 (34 \#751)}

\bibitem{ONeil83}
\bysame, \emph{Semi-{R}iemannian geometry}, Pure and Applied Mathematics, vol.
  103, Academic Press Inc. [Harcourt Brace Jovanovich Publishers], New York,
  1983, With applications to relativity. \MR{719023 (85f:53002)}

\bibitem{BeBer14}
Bo~Berndtsson Robert~Berman, \emph{Convexity of the {K}-energy on the space of
  {K}\"{a}hler metrics and uniqueness of extremal metrics}, arXiv:1405.0401v.3,
  2014.

\bibitem{Sem92}
Stephen Semmes, \emph{Complex {M}onge-{A}mp\`ere and symplectic manifolds},
  Amer. J. Math. \textbf{114} (1992), no.~3, 495--550. \MR{1165352 (94h:32022)}

\bibitem{Spa11}
James Sparks, \emph{Sasaki-{E}instein manifolds}, Surveys in differential
  geometry. {V}olume {XVI}. {G}eometry of special holonomy and related topics,
  Surv. Differ. Geom., vol.~16, Int. Press, Somerville, MA, 2011, pp.~265--324.
  \MR{2893680 (2012k:53082)}

\bibitem{Sto09b}
Jacopo Stoppa, \emph{Twisted constant scalar curvature {K}\"ahler metrics and
  {K}\"ahler slope stability}, J. Differential Geom. \textbf{83} (2009), no.~3,
  663--691. \MR{2581360 (2011f:32052)}

\bibitem{vC12b}
Craig van Coevering, \emph{Stability of {S}asaki-extremal metrics under complex
  deformations}, Int. Math. Res. Not. IMRN (2013), no.~24, 5527--5570.

\bibitem{CLP14}
Mihai~P\u{a}un Xiu Xiong~Chen, Long~li, \emph{Approximation of weak geodesics
  and subharmonicity of {M}abuchi energy}, arXiv:1409.7896, 2014.

\bibitem{CPZ15}
Yu~Zeng Xiu Xiong~Chen, Mihai~P\u{a}un, \emph{On deformation of extremal
  metrics}, arXiv:1506.01290v2, 2015.

\bibitem{Yau93}
Shing-Tung Yau, \emph{Open problems in geometry}, Differential geometry:
  partial differential equations on manifolds ({L}os {A}ngeles, {CA}, 1990),
  Proc. Sympos. Pure Math., vol.~54, Amer. Math. Soc., Providence, RI, 1993,
  pp.~1--28. \MR{1216573 (94k:53001)}

\end{thebibliography}

\providecommand{\bysame}{\leavevmode\hbox to3em{\hrulefill}\thinspace}
\providecommand{\MR}{\relax\ifhmode\unskip\space\fi MR }
\providecommand{\MRhref}[2]{%
  \href{http://www.ams.org/mathscinet-getitem?mr=#1}{#2}
}
\providecommand{\href}[2]{#2}

\end{document}